\documentclass[10pt, a4paper, reqno, english]{amsart}
\pdfoutput=1
\usepackage{amsfonts, amsthm, amsmath, amssymb, mathrsfs}

\usepackage[utf8]{inputenc}
\usepackage{lmodern}

\usepackage[margin=3.5cm]{geometry}
\usepackage{enumitem}
\usepackage[draft=false,hidelinks]{hyperref}
\usepackage{graphics}

\usepackage{tikz}
\usetikzlibrary{shapes.geometric}

\theoremstyle{plain}
\newtheorem{theorem}{Theorem}
\newtheorem{prop}[theorem]{Proposition}
\newtheorem{lemma}[theorem]{Lemma}
\theoremstyle{remark}
\newtheorem{remark}[theorem]{Remark}

\newcommand\ee{\boldsymbol{\eta}}
\newcommand\eei{\ee^{(i)}}
\newcommand\e{\eta}
\newcommand\eps{\varepsilon}

\newcommand{\FF}{\mathbb{F}}
\newcommand{\PP}{\mathbb{P}}
\newcommand{\QQ}{\mathbb{Q}}
\newcommand{\VV}{\mathbb{V}}
\newcommand{\Qbar}{\overline{\QQ}}
\newcommand{\RR}{\mathbb{R}}
\newcommand{\AAA}{\mathbb{A}}

\newcommand{\ZZ}{\mathbb{Z}}
\newcommand{\RRnn}{\RR_{\ge{}0}}

\newcommand{\GGmZ}{\mathbb{G}_{\mathrm{m},\ZZ}} 

\newcommand\Os{\mathcal{O}}
\newcommand\Us{\mathcal{U}}
\newcommand\Ss{\mathcal{S}}
\newcommand\Zs{\mathcal{Z}}
\newcommand\Hs{\mathcal{H}}
\newcommand\Ys{\mathcal{Y}}
\newcommand\Ls{\mathcal{L}}
\newcommand{\Can}{\mathcal{C}^{\mathrm{an}}}

\newcommand\tU{\widetilde{U}}
\newcommand\tV{\widetilde{V}}
\newcommand\tS{\widetilde{S}}
\newcommand\tH{\widetilde{H}}

\newcommand\tSs{\widetilde{\Ss}}
\newcommand\tUs{\widetilde{\mathcal{U}}}

\newcommand{\Ab}{\mathbf{A}}
\newcommand{\ADE}{\mathbf{ADE}}

\DeclareMathOperator{\vol}{vol}
\DeclareMathOperator{\Pic}{Pic}
\DeclareMathOperator{\Eff}{Eff}
\DeclareMathOperator{\Spec}{Spec}
\DeclareMathOperator{\rk}{rk}

\newcommand{\fin}{\mathrm{fin}}
\newcommand{\irr}{\mathrm{irr}}

\newcommand{\abs}[1]{\left\lvert#1\right\rvert}
\newcommand{\norm}[1]{\left\lVert#1\right\rVert}
\newcommand\bigwhere[2]{\left\{#1\
    \left|\ \begin{aligned}#2\end{aligned}\right.\right\}}
\newcommand\where[2]{\{#1 \mid{} #2\}}
\newcommand*\diff{\mathop{}\!\mathrm{d}}

\begin{document}

\title{Integral points on singular del Pezzo surfaces}

\author{Ulrich Derenthal}
\email{derenthal@math.uni-hannover.de}
\address{Institut f\"ur Algebra, Zahlentheorie und Diskrete
  Mathematik, Leibniz Universit\"at Hannover,
  Welfengarten 1, 30167 Hannover, Germany}

\author{Florian Wilsch}
\email{florian.wilsch@ist.ac.at}
\address{Institute of Science and Technology Austria,
  Am Campus 1, 3400 Klosterneuburg, Austria}

\date{November 10, 2022}


\setcounter{tocdepth}{1}

\begin{abstract}
  In order to study integral points of bounded log-anticanonical height on
  weak del Pezzo surfaces, we classify weak del Pezzo pairs.  As a
  representative example, we consider a quartic del Pezzo surface of
  singularity type $\Ab_1+\Ab_3$ and prove an analogue of Manin's conjecture
  for integral points with respect to its singularities and its lines.
\end{abstract}

\maketitle

\tableofcontents

\section{Introduction}

Del Pezzo surfaces over $\QQ$ often contain infinitely many rational points.
Over the past 20 years, Manin's conjecture~\cite{MR974910,MR2019019} for the
asymptotic behavior of the number of rational points of bounded anticanonical
height has been confirmed for some smooth and many singular del Pezzo surfaces
(see~\cite{MR1909606,MR2332351,MR2874644} for some milestones
and~\cite[\S~6.4.1]{MR3307753} for many further references), in most cases
using universal torsors, often combined with advanced analytic techniques.

In recent years, a conjectural framework for the density of \emph{integral}
points has emerged in the work of Chambert-Loir and Tschinkel~\cite{MR2740045}.
The purpose of this paper is to
initiate a systematic investigation of integral points of bounded height on
del Pezzo surfaces.
Only a few of them are covered by general results for equivariant compactifications of vector
groups~\cite{MR2999313} or the incomplete work on toric
varieties~\cite{arXiv:1006.3345v2} (see also~\cite{wilsch-toric}); most del Pezzo surfaces are out of reach of this
\emph{harmonic analyis approach} since they are not equivariant compactifications of
algebraic groups~\cite{MR2753646,MR3333982}. Del Pezzo surfaces are
inaccessible to the \emph{circle method}, which gives asymptotic formulas for
integral points only on high-dimensional complete
intersections~\cite{MR0150129},~\cite[\S~5.4]{MR2740045}. Therefore, we adapt
the \emph{universal torsor method} to integral points in order to confirm new cases
of an integral analogue of Manin's conjecture.
See also~\cite{arXiv:1901.08503} for a
three-dimensional example.

As rational and integral points coincide on a projective variety $X$, the
study of the latter becomes interesting on its own on an integral model of the
complement $X\setminus Z$ of an appropriate \emph{boundary} $Z$. Our first
result (Theorem~\ref{thm:classification} in Section~\ref{sec:classification})
is a general treatment of possible boundaries on singular del Pezzo surfaces
of low degree. For singular cubic surfaces, $Z$ must be an
$\mathbf{A}$-singularity; for singular quartic del Pezzo surfaces, $Z$ must be
an $\mathbf{A}$-singularity or a line passing only through
$\mathbf{A}$-singularities. Furthermore, $\Ab_1$-singularities behave
differently than other $\mathbf{A}$-singularities.

Therefore, a good starting point seems to be a
quartic del Pezzo surface that contains an $\Ab_1$- and an
$\Ab_3$-singularity and three lines, which is neither toric \cite[Remark~6]{MR3180592} nor a compactification of $\mathbb{G}_\mathrm{a}^2$~\cite{MR2753646}.
For each boundary $Z$ admissible in the sense of Theorem~\ref{thm:classification},
we get an associated counting problem and prove an asymptotic formula of the shape
\[
  c B (\log B)^{b-1}
\]
(Theorem~\ref{thm:main_concrete}), encountering a range of different phenomena
when dealing with the different types of boundary. These asymptotic formulas
admit a geometric interpretation (Theorem~\ref{thm:main_abstract}). In
particular, the leading constant $c$ consists of Tamagawa numbers as defined
in~\cite{MR2740045} and combinatorial constants (analogous to the constant
$\alpha$ defined by Peyre for rational points) as defined
in~\cite{arXiv:1006.3345v2} for toric varieties and studied in greater
generality in~\cite{wilsch-toric}; this is the first result applying this
combinatorial construction in a nontoric setting.

\subsection{The counting problem}

Let $S \subset \PP^4_\QQ$ be the quartic del
Pezzo surface  defined by
\begin{equation}\label{eq:def_eq}
  x_0^2+x_0x_3+x_2x_4=x_1x_3-x_2^2=0
\end{equation}
over $\QQ$, with an $\Ab_1$-singularity $Q_1=(0:1:0:0:0)$ and an
$\Ab_3$-singularity $Q_2=(0:0:0:0:1)$.
Let $\Ss \subset \PP^4_\ZZ$ be its integral model defined by the same
equations over $\ZZ$.

The closure of every rational point $P \in S(\QQ)$ is an integral point
$\overline{P} \in \Ss(\ZZ)$; both are represented (uniquely up to sign) by
coprime
$(x_0,\dots,x_4) \in \ZZ^5 \setminus \{0\}$ satisfying the defining
equations~\eqref{eq:def_eq}.
Recall that studying integral points becomes interesting only when we
choose a \emph{boundary} $\Zs$ to
consider integral points on $\Ss \setminus \Zs$, and that the types of
boundaries in Theorem~\ref{thm:classification} for our case are the
singularities and the lines; we start with the former.
To do so, let $Z_1=Q_1$, $Z_2=Q_2$; in addition to these, we study the boundary
$Z_3=Q_1\cup Q_2$, which goes beyond the setting of weak del Pezzo pairs described in the beginning of
the following section. Let $\Zs_i=\overline{Z_i}$,
and $\Us_i = \Ss \setminus \Zs_i$. Hence, $\overline{P}$ lies in
$\Us_3(\ZZ)$, say, if and only if it is does not reduce to one of the singularities
modulo any prime $p$. In other words, a representative $(x_0,\dots,x_4)$ of a
point in $\Us_i(\ZZ)$ satisfies the \emph{integrality condition}
\begin{equation}\label{eq:gcd}
  \begin{aligned}
    \gcd(x_0,x_2,x_3,x_4)=1,\qquad &\text{if $i=1$},\\
    \gcd(x_0,x_1,x_2,x_3)=1,\qquad &\text{if $i=2$, or}\\
    \gcd(x_0,x_2,x_3,x_4)=1\quad\text{and}\quad \gcd(x_0,x_1,x_2,x_3)=1,\qquad &\text{if $i=3$}.
  \end{aligned}
\end{equation}

Since the sets $\Us_i(\ZZ)$ of integral points are clearly infinite,
we consider integral points of bounded height. We work with the
\emph{height functions}
\begin{equation}\label{eq:height}
  \begin{aligned}
    H_1(\overline{P})
    &=\max\{\abs{x_0},\abs{x_2},\abs{x_3},\abs{x_4}\},\\
    H_2(\overline{P})
    &=\max\{\abs{x_0},\abs{x_1},\abs{x_2},\abs{x_3}\},\qquad \text{and}\\
    H_3(\overline{P})
    &=\max\{\abs{x_0},\abs{x_2},\abs{x_3}, \min\{\abs{x_1},\abs{x_4}\}\},
  \end{aligned}
\end{equation}
because they can be interpreted as \emph{log-anticanonical heights} 
on a minimal desingularization, as we shall see below (Lemma~\ref{lem:heights}).

It turns out that the number of integral points of bounded height is
dominated by the integral points on the three lines
\begin{equation}\label{eq:lines}
  L_1=\{x_0=x_2=x_3=0\},\ 
  L_2=\{x_0=x_1=x_2=0\},\ 
  L_3=\{x_0+x_3=x_1=x_2=0\};
\end{equation}
in fact, there are infinitely many integral points of height $1$ on some of
them.
Therefore, we count integral points only in their complement
$V = S \setminus \{x_2=0\}$.
Hence, we are interested in the asymptotic behavior of
\begin{equation}\label{eq:def-N_i}
  N_i(B) = \# \{\overline{P} \in \Us_i(\ZZ) \cap V(\QQ) \mid H_i(\overline{P}) \le B\},
\end{equation}
the number of integral points of bounded log-anticanonical height
that are not contained in the lines.  Explicitly, this is
\begin{equation}\label{eq:description-N_i}
  N_i(B) = \# \{(x_0,\dots,x_4) \in \ZZ^5 \setminus \{0\} \mid
  \eqref{eq:def_eq},\ \eqref{eq:gcd},\ x_2 > 0,\ H_i(x_0:\dots:x_4)
  \le B\}.
\end{equation}

Recall that the second type of boundary is a line, resulting in
$Z_4=L_1,Z_5 = L_2, Z_6 = L_3$ with the notation in~\eqref{eq:lines}.
Let $\Zs_i=\overline{Z_i}$ in $\Ss$, and $\Us_i=\Ss \setminus \Zs_i$ for
$i=4,5,6$.
Analogously to the first three cases, a point $(x_0:\dots:x_4)\in S$ with coprime $x_0,\dots,x_4 \in \ZZ$
lies in $\Us_i(\ZZ)$ if and only if
\begin{equation}\label{eq:gcd-lines}
  \begin{aligned}
    \gcd(x_0,x_2,x_3)=1, \qquad &\text{if $i=4$},\\
    \gcd(x_0,x_1,x_2)=1, \qquad &\text{if $i=5$, or}\\
    \gcd(x_0+x_3,x_1,x_2)=1, \qquad &\text{if $i=6$}.
  \end{aligned}
\end{equation}
We work with the heights
\begin{align*}
  H_4(\overline{P}) &= \max\{|x_0|,|x_2|,|x_3|\},\\      
  H_5(\overline{P}) &= \max\{|x_0|,|x_1|,|x_2|\},\qquad \text{and}\\
  H_6(\overline{P}) &= \max\{|x_0+x_3|,|x_1|,|x_2|\},
\end{align*}
which will again turn out to be log-anticanonical on a minimal desingularization.
Let $N_i(B)$ for $i=4,5,6$ be defined as in~\eqref{eq:def-N_i}. They satisfy
descriptions as in~\eqref{eq:description-N_i}, with the integrality
condition~\eqref{eq:gcd} replaced by~\eqref{eq:gcd-lines}.

Our second result consists of asymptotic formulas for these counting problems:
\begin{theorem}\label{thm:main_concrete}
  As $B \to \infty$, we have
  \begin{align*}
    N_1(B) &=
    \frac{13}{4320}
    \left(\prod_p \left(1-\frac 1 p\right)^5\left(1+\frac 5 p\right)\right)
    B (\log B)^5 + O(B (\log B)^4 \log\log B),
    \\
    N_2(B) & = 
    \frac{1}{32}
    \left(\prod_p \left(1-\frac 1 p\right)^3\left(1+\frac 3 p\right)\right)
    B (\log B)^4 + O(B(\log B)^3 \log\log B),\\
    N_3(B) & = 
    \frac 1 8
    \left(\prod_p \left(1-\frac 1 p\right)^2\left(1+\frac 2 p-\frac 1 {p^2}\right)\right)
    B (\log B)^3 + O(B(\log B)^2 \log\log B), \\
    N_4(B) &=  2 
    \left(\prod_p \left(1-\frac 1 p\right)\left(1+\frac 1 p\right)\right)
    B (\log B)^2 + O(B \log B \log\log B), \quad \text{and}\\
    N_5(B) = N_6(B) &= \frac{7}{24} 
    \left(\prod_p \left(1-\frac 1 p\right)^2\left(1+\frac 2 p\right)\right)
    B (\log B)^3 + O(B(\log B)^2 \log\log B).
  \end{align*}
\end{theorem}

Cases 5 and 6 are symmetric: the involutive automorphism
\begin{equation}\label{eq:symmetry}
  (x_0,x_1,x_2,x_3,x_4)\mapsto (x_0+x_3,-x_1,x_2,-x_3,x_4)
\end{equation}
of $S$ exchanges the lines $L_2$ and $L_3$ and the height functions $H_5$ and
$H_6$, while leaving $V=S\setminus\{x_2=0\}$ invariant, whence
$N_5(B)=N_6(B)$.

\subsection{The expected asymptotic formula}\label{sec:expected}

Similarly to the case of rational points~\cite{MR1679843,MR2019019}, our
asymptotic formulas for the number of integral points of bounded height should
be interpreted on a desingularization $\rho \colon \tS \to S$. Here, $\tS$ is
a \emph{weak del Pezzo surface}, that is, a smooth projective surface whose
anticanonical bundle $\omega_{\tS}^\vee$ is big and nef (but not ample in our
case).

To interpret the number of points on $\Us_i = \Ss \setminus \Zs_i$, we study a
desingularization $\tU_i = \tS \setminus D_i$ of $U_i$, where
$D_i=\rho^{-1}(Z_i)$ is a reduced effective divisor with strict normal
crossings.
In the context of integral points, the log-anticanonical bundle
$\omega_{\tS}(D_i)^\vee$ assumes the role of the anticanonical bundle.
From this point of view, Theorem~\ref{thm:main_concrete} can be interpreted in
the framework described in~\cite{MR2740045}.

The minimal desingularization $\rho: \tS \to S$ is an iterated blow-up of
$\PP^2_\QQ$ in five points. The analogous blow-up of $\PP^2_\ZZ$ results in an
integral model $\rho\colon \tSs \to \Ss$ (see section~Section~\ref{sec:torsor} for more details).  Then $D_1$,
$D_2$ are the divisors above $Q_1,Q_2$, respectively, and $D_3=D_1+D_2$ is the one over
both; see Figure~\ref{fig:clemens_complex} for their dual graph (Dynkin
diagram). Our discussion is simplified by the fact that the pairs $(\tS,D_i)$
are split, in the sense that $\Pic\tS \to \Pic\tS_{\Qbar}$ is an
isomorphism and~\cite[Definition~1.6]{MR3732687} holds, and by the fact that we
are working over $\QQ$.
Let $\tU_i,\tUs_i$ be the complement of $D_i,\overline{D_i}$ in $\tS,\tSs$,
respectively, where $\overline{D_i}$ is the Zariski closure of $D_i$ in
$\tSs$.  The preimage of the complement $V$ of the lines on $S$ is the
complemenent $\tV$ of all negative curves on $\tS$.

This leads to the reinterpretation of our counting problem as
\begin{equation*}
  N_i(B) = \#\{\overline{P} \in \tUs_i(\ZZ) \cap \tV(\QQ) \mid H_i(\rho(\overline{P})) \le B\}
\end{equation*}
on the minimal desingularization, and we prove in Lemma~\ref{lem:heights} that
$H_i \circ \rho$ is a log-anticanonical height function on
$\tUs_i(\ZZ) \cap \tV(\QQ)$. Note that the log-anticanonical bundle
$\omega_{\tS}(D_i)^\vee$ is big and nef for $i=1,2,4,5,6$, but big and
\emph{not} nef for $i=3$ (Lemma~\ref{lem:big-and-nef}); the unusual shape of
$H_3$ is clearly related to this.

\medskip

From the shape of asymptotic formulas in previous results~\cite{arXiv:1006.3345v2,MR2999313,MR3117310,arXiv:1901.08503} and
the study of volume asymptotics in~\cite{MR2740045}, we expect that
\begin{equation*}
  N_i(B) \sim c_{i,\fin}c_{i,\infty} B (\log B)^{b_i-1},
\end{equation*}
where the leading constant can be decomposed into a finite part $c_{i,\fin}$ and an archimedean
part $c_{i,\infty}$ that we shall describe and determine in Section~\ref{sec:leading-constant}
precisely.

The finite part
\begin{equation}\label{eq:c_fin}
c_{i,\fin} = \prod_p \left(1-\frac 1 p \right)^{\rk\Pic\tU_i} \tau_{(\tS,D_i), p}(\tUs_i(\ZZ_p)),
\end{equation}
which behaves similarly as in the case of
rational points, is defined as an Euler product of convergence factors and $p$-adic Tamagawa numbers.
We compute the latter as $p$-adic integrals over $\tUs_i(\ZZ_p)$
(Lemma~\ref{lem:surface-finite-volumes}); they turn out to be simply
$\#\tUs_i(\FF_p)/p^{\dim S}$.
This reflects the fact that integral points should be distributed evenly in the set
$\tUs_i(\ZZ_p)$, which has positive and finite volume with respect to the modified
Tamagawa measure $\tau_{(\tS,D_i),p}$ defined in~\cite{MR2740045}. (However, we do not prove such an equidistribution result here.)

\begin{figure}[t]
  \begin{center}
    \medskip
    \begin{tikzpicture}
      \node (E7) at (-4,0) [draw, shape=circle, fill=black, scale=.3]{};
      \node (E3) at (-2,0) [draw, shape=circle, fill=black, scale=.3]{};
      \node (E4) at (0,0)  [draw, shape=circle, fill=black, scale=.3]{};
      \node (E6) at (2,0)  [draw, shape=circle, fill=black, scale=.3]{};

      \draw[thick] (E3) -- (E4) node[pos=0.5, anchor=north]{$A_1$};
      \draw[thick] (E4) -- (E6) node[pos=0.5, anchor=north]{$A_2$};
    \end{tikzpicture}
    \caption{The Clemens complex of $D_3$ is the disjoint union of those of
      $D_1$ (left) and $D_2$ (right). It is the Dynkin diagram of the
      $\Ab_1$- and $\Ab_3$-singularities $Q_1,Q_2$.}\label{fig:clemens_complex}
  \end{center}
\end{figure}
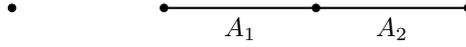

On the other hand, 100\% of the integral points are arbitrarily close to the
boundary with respect to the real-analytic topology, ordered by height. This
makes the analysis of $c_{i,\infty}$ much more delicate than for rational points.
More precisely, the points close to the minimal strata of the boundary---that is,
the intersection of a maximal set of intersecting components of $D_i$---should
dominate the counting function. These strata are encoded in the \emph{(analytic) Clemens
  complex} $\Can_\RR(D_i)$.  For a split surface, the vertices of this Clemens
complex correspond to the irreducible components of the boundary divisor
$D_i$, and there is an edge for each intersection point of two divisors.
The
archimedean constant
\begin{equation}\label{eq:c_infty}
  c_{i,\infty} = \sum_A \alpha_{i,A} \tau_{i,D_A,\infty}(D_A(\RR))\\
\end{equation}
is a sum over the faces $A$ of maximal dimension of
the Clemens complex, which correspond to the minimal strata $D_A$ of $D_i$.
For each maximal-dimensional face $A$, we have a product of a rational factor
$\alpha_{i,A}$ and an archimedean Tamagawa number $\tau_{i,D_A,\infty}(D_A(\RR))$
coming from a \emph{residue measure} as defined in~\cite{MR2740045}.
This measure can be interpreted as a real density, which is supported on
$D_A(\RR)$ and should measure the distribution of points in neighborhoods of open
subsets of $D_A(\RR)$.
From another point of view, the set $\tS(\RR)$ has infinite volume with respect
to a modified measure $\tau_{(\tS,D_i),\infty}$ as above,
and $\tau_{i,D_A,\infty}(D_A(\RR))$ appears in the leading constant of the
asymptotic volume of \emph{height balls} with respect to said
measure~(cf.~\cite[Propositions~2.5.1, 4.2.4]{MR2740045}).

In the first case, the Clemens complex consists of only one vertex
corresponding to the boundary divisor above the $\Ab_1$-singularity $Q_1$, and
integral points accumulate near it (Figure~\ref{fig:points_U1}). In the second
and third case, the maximal dimensional faces $A_1,A_2$ of the Clemens complex
correspond to the two intersection points $D_{A_1},D_{A_2}$ of the divisors
above the $\Ab_3$-singularity, and ``most'' integral points are very close
to these two intersection points
(Figure~\ref{fig:points_U2}). Correspondingly, the archimedean Tamagawa number
is the volume of the boundary divisor in the first case, and it is the volume
of the two intersection points in the second and third case. In the remaining cases, it similarly is a
volume of intersection points (Lemma~\ref{lem:arch-volumes}).

The rational factor $\alpha_{i,A}$ is particularly interesting in our
examples.  It is introduced in~\cite{arXiv:1006.3345v2} for toric varieties and generalized in~\cite{wilsch-toric}
to be 
\begin{equation}\label{eq:alpha}
  \alpha_{i,A} = \vol\{x \in (\Eff\tU_{i,A})^\vee \mid
  \langle x, \omega_{\tS}(D_i)^\vee|_{\tU_{i,A}}\rangle = 1\},
\end{equation}
where $\tU_{i,A}$ is the subvariety consisting of $\tU_i$ and the divisors
corresponding to $A$.
For vector groups~\cite{MR2999313} and wonderful compactifications~\cite{MR3117310}, the
effective cone is generated by the boundary divisors and simplicial, which
makes the treatment of this factor easy. In~\cite{arXiv:1901.08503},
it behaves similarly as Peyre's $\alpha$ for projective varieties
since the boundary has just one component; it is also much simpler since the
Picard number is $2$.  Our second and following cases behave in a different
way since the Clemens complex is not a simplex, providing the first nontrivial
treatment of this factor for a nontoric variety.
Here, it turns out that the resulting polytopes for the different maximal faces fit together to one
polytope whose volume appears in the leading constant of the counting problem
(Lemma~\ref{lem:arch-constants}). In case $4$, one of the polytopes has volume $0$, making this an example for the obstruction~\cite[Theorem~2.4.1~(i)]{wilsch-toric} to the existence of integral points near the corresponding minimal stratum of the boundary (Remark~\ref{rmk:obstruction}).

\begin{figure}[t]
  \begin{center}
    \includegraphics{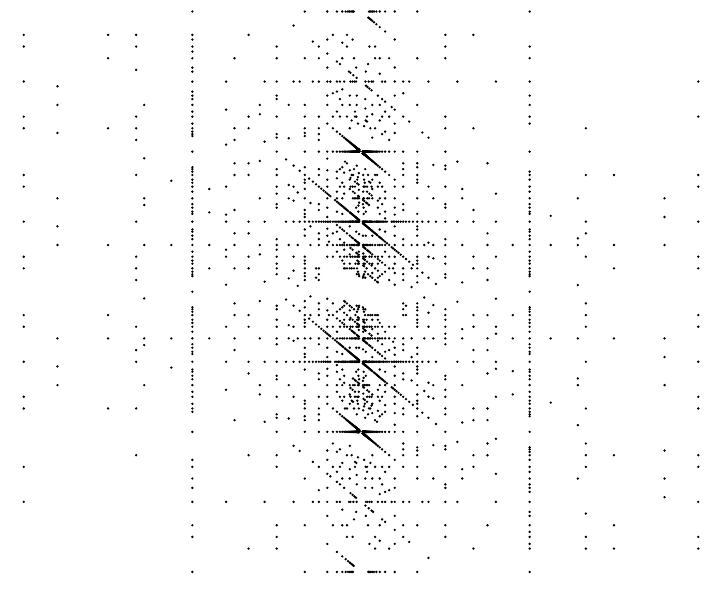}
  \end{center}
  \caption{Integral points on $\tUs_1$ of height $\le 90$. The boundary
    divisor is the central vertical line. Some horizontal and diagonal lines
    look accumulating, but in fact are not: They contain $\sim c^\prime B$
    points, which is less than the $c B(\log B)^5$ points on $U$; the
    constants $c^\prime$ can however be up to $2$, while the constant $c$ in
    our main theorem is numerically $\approx .0003$.}\label{fig:points_U1}
\end{figure}

\begin{figure}[t]
  \begin{center}
    \includegraphics{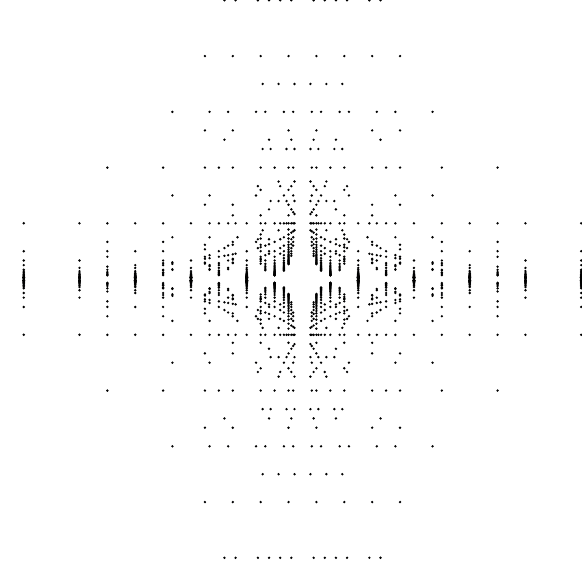}
    \includegraphics{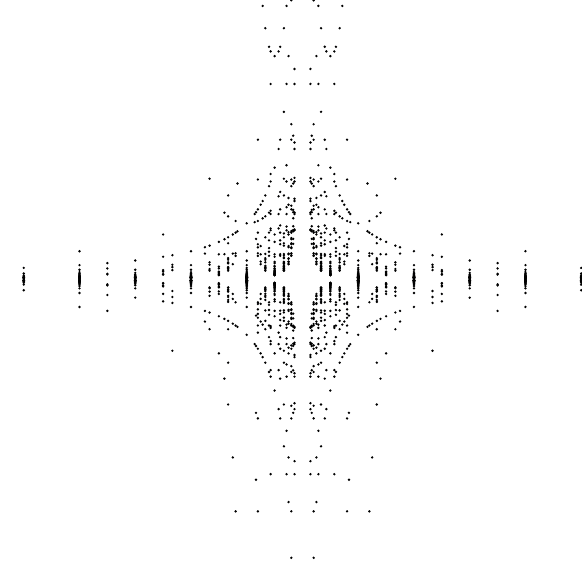}
  \end{center}
  \caption{Integral points on
    $\tUs_2$ of height $\le 60$, in neighborhoods of $D_{A_1}$ (left) and
    $D_{A_2}$ (right). Most points are close to the three boundary divisors,
    which are the central horizontal line and two vertical lines here.}\label{fig:points_U2}
\end{figure}

The exponent of $\log B$ is expected to be $b_i-1$, where
\begin{equation}\label{eq:exponent}
  b_i = \rk\Pic\tU_i  - \rk \QQ[U_i]/\QQ^\times + \dim \Can_\RR(D_i) + 1.
\end{equation}
Here, $\dim \Can_\RR(D_i) + 1$ is the maximal number of components of the
boundary divisor $D_i$ that meet in the same point, and $\QQ[U_i]^\times = \QQ^\times$ in each case. 
While the obstruction described in~\cite{wilsch-toric} can lead to this number being smaller than expected if it affects all maximal-dimensional faces of the Clemens complex, this does not happen in our fourth case as there are three unobstructed faces remaining.

We can reformulate Theorem~\ref{thm:main_concrete} as follows:
\begin{theorem}\label{thm:main_abstract}
  For $i \in \{1,\dots,6\}$, we have
  \begin{equation}\label{eq:main-formula}
    N_i(B) = c_{i,\infty} c_{i,\fin} B(\log B)^{b_i-1}(1+o(1))
  \end{equation}
  as $B \to \infty$, where the constants $c_{i,\infty}$, $c_{i,\fin}$, and $b_i$
  are as in~\eqref{eq:c_fin},~\eqref{eq:c_infty}, and~\eqref{eq:exponent}, respectively.
\end{theorem}

This confirms the expectations extracted from~\cite{MR2740045,arXiv:1006.3345v2,wilsch-toric}.

\subsection{Strategy of the proof}

In Section~\ref{sec:classification}, we define and classify \emph{weak del
  Pezzo pairs} $(\tS,D)$, which have big and nef log-anticanonical bundle
$\omega_{\tS}(D)^\vee$
(Theorem~\ref{thm:classification}).

In Section~\ref{sec:torsor}, we describe a universal torsor on the
minimal desingularization of $S$, we show that our height functions
are log-anticanonical, and we describe them in terms of Cox
coordinates. This leads to a completely explicit counting problem on
the universal torsor (Lemma~\ref{lem:param}), with a
$(2^{\rk\Pic\tU_i}:1)$-map to our set of integral points of bounded
height: roughly, the torsor variables corresponding to the boundary
divisors must be $\pm 1$, and for big and base point free (whence nef)
log-anticanonical class,
the height function $H_i$ is given by monomials in the Cox ring of
log-anticanonical degree. The third case seems to be one of the first
examples of the universal torsor method with respect to a height for a divisor
class that is big and \emph{not} nef.

In Section~\ref{sec:counting}, we estimate the number of points in our
counting problem on the universal torsor using analytic techniques. Here, we
approximate summations over the torsor variables by real integrals
$V_{i,0}(B)$; the coprimality conditions lead to an Euler product that agrees
with $c_{i,\fin}$ (Lemmas~\ref{lem:first_summation} and~\ref{lem:surface-result-of-count}).
This step is similar to the case of
rational points treated in~\cite{MR2520770}; hence, we shall be very brief.

In Section~\ref{sec:volumes}, to complete the proof of
Theorem~\ref{thm:main_concrete}, our goal is to transform $V_{i,0}(B)$ into
$2^{\rk\Pic\tU_i} C_i B(\log B)^{b_i-1}$, where $C_i$ is the product of a the
volume of a polytope (which turns out to be $\sum \alpha_{i,A}$) and a real
density (which agrees with the archimedean Tamagawa numbers
$\tau_{i,D_A,\infty}(D_A(\RR))$), up to a negligible error term.
In the first case, there is
a complication due to an inhomogeneous expression (with respect to the grading
by the Picard group) in the domain of $V_{1,0}$ (Lemma~\ref{lem:adjust_V1}
and more importantly Lemma~\ref{lem:treat_W});
here, a subtle estimation is necessary. In the third case, we modify the
height function $H_3$ to $H_3'$ (which coincides essentially with $H_2$) as in
Lemma~\ref{lem:adjust_V3}. These extra complications have never appeared in
the universal torsor method for rational points; we believe that they are
typical for integral points and nonnef heights.

In Section~\ref{sec:leading-constant}, we prove
Theorem~\ref{thm:main_abstract} by explicitly computing the expected
constants discussed in Section~\ref{sec:expected}.

\subsection*{Acknowledgements}

The first author was partly supported by grant DE 1646/4-2 of the
Deutsche Forschungsgemeinschaft. The second author was partly
supported by FWF grant P~32428-N35 and conducted part of this work as
a guest at the Institut de Mathématiques de Jussieu--Paris Rive Gauche
invited by Antoine Chambert-Loir and funded by DAAD\@. The authors
thank the anonymous referees for their useful remarks and suggestions.

\section{Classification of weak del Pezzo pairs}\label{sec:classification}

For us, a \emph{weak del Pezzo pair} $(\tS,D)$ consists of a smooth
projective surface $\tS$ with a reduced effective divisor $D$ with strict normal
crossings such that the log-anticanonical bundle $\omega_{\tS}(D)^\vee$ is big
and nef. The aim of this section is to study the possible choices of divisors $D$ on a weak del Pezzo surface $\tS$ that render the pair
$(\tS,D)$ weak del Pezzo in this sense.

\begin{remark}
  Considering \emph{pairs} $(X,D)$ is standard when studying integral points:
  While rational and integral points coincide on complete
  varieties as a consequence of the valuative criterion for properness, the
  study of integral points becomes a distinct problem on an integral model $\Us$ of a
  noncomplete variety $U$.
  Then one passes to a compactification, more precisely, a smooth projective
  variety $X$ containing $U$ such that the boundary $D = X \setminus U$ is a
  reduced effective divisor with strict normal crossings. In particular, the
  pair $(X,D)$ is smooth and divisorially log terminal.

  The goal is then to count the number of points on $\Us$ of bounded
  log-anticanonical height (that is, with respect to $\omega_X(D)^\vee$),
  excluding any strict subvarieties (or, more generally, \emph{thin} subsets)
  whose points would contribute to the main term. Setting $D=0$ then recovers the
  setting of Manin's conjecture on rational points.
\end{remark}

\begin{remark}
  In its original form~\cite{MR974910,MR1340296}, Manin's conjecture makes a
  prediction about the number of rational points on \emph{smooth Fano
    varieties}: smooth projective varieties whose anticanonical bundle is
  ample.  These conditions can be relaxed, for example only requiring that the anticanonical be big and nef, viz.\ to \emph{weak Fano varieties} and the $2$-dimensional varieties therof, \emph{weak del Pezzo surfaces}. Weak del Pezzo surfaces
  $\tS$ are precisely the smooth del Pezzo surfaces $\tS = S$ and the minimal
  desingularizations $\rho\colon \tS\to S$ of del Pezzo surfaces with only
  $\ADE$-singularities~\cite{MR579026_new}.

  Since $\rho$ is a crepant resolution---that is,
  $\omega_{\tS} = \rho^* \omega_S$---counting points on $S$ of bounded
  anticanonical height amounts to counting points on $\tS$ of bounded
  anticanonical height after excluding points on the exceptional locus.
  By~\cite{MR1679843,MR2019019}, an asymptotic formula for the number of rational
  points on $S$ should be interpreted in terms of its minimal
  desingularization $\tS$; for example, the Picard rank $\rho$ of $\tS$ appears in
  the expected asymptotic formula. The number of rational points of bounded
  height has been shown to conform to the same prediction as in Manin's
  conjecture for many weak del Pezzo surfaces (see the references
  in~\cite[\S~6.4.1]{MR3307753}).

  Generalizing the question even further, it suffices to assume that the anticanonical bundle is big to guarantee that the number of rational points of bounded anticanonical height outside a suitable divisor is finite.
  Adding some conditions that make Peyre's constant well-defined leads to the notion of an \emph{almost Fano} variety~\cite[Définition~3.1]{MR2019019}, for which it makes sense to ask whether Manin's conjecture holds. While this is known to be the case for some of them, Lehmann, Sengupta, and Tanimoto showed that one cannot expect the conjecture to be true in general in this widest setting~\cite[Remark~1.1, Example~5.17]{arxiv.1805.10580}.
\end{remark}

To simplify the
exposition, let $\tS$ be a weak del Pezzo surface whose degree $d$ is at most $7$.  Let
$D=\sum_{\alpha \in \mathcal{A}} D_\alpha \subset \tS$ be a reduced and
effective divisor with strict normal crossings and irreducible components
$D_\alpha$.

\begin{lemma}\label{lem:nef-criterion}
    The log-anticanonical bundle $\omega_{\tS}(D)^\vee$ is
  nef if and only if all of the following conditions hold:
  \begin{enumerate}[label = (\roman*)]
    \item\label{enum:cond--2} If $E$ is a $(-2)$-curve and $D_\alpha.E>0$ for
      some $\alpha \in \mathcal{A}$, then $E \subset D$.
    \item\label{enum:cond--1} If $E$ is an arbitrary negative curve meeting
      two different $D_\alpha,D_\beta$ or one $D_\gamma \subset D$ with
      multiplicity $D_\gamma.E\ge 2$ (with
      $\alpha,\beta,\gamma \in \mathcal{A}$), then $E \subset D$.
    \item\label{enum:cond-no-three} If $E$ is a negative curve, then
    \[
      \sum_{\substack{\alpha \in \mathcal{A} \\ D_\alpha \ne E }} D_\alpha . E \le 2.
    \]
  \end{enumerate}
\end{lemma}
\begin{proof}
  Recall that a divisor is nef if its intersection with all negative curves is nonnegative.
  If $E$ is a $(-2)$-curve, then 
  \[
    (-K-D).E = -K.E-D.E
    = 0 + 2\delta_{E \subset D}
    - \sum_{\substack{\alpha \in \mathcal{A} \\ D_\alpha \ne E }}D_\alpha . E,
  \]
  and this number is nonnegative if and only if~\ref{enum:cond--2} and~\ref{enum:cond-no-three} hold for $E$.
  If $E$ is a $(-1)$-curve, then 
  \[
    (-K-D).E = -K.E-D.E
    = 1 + \delta_{E \in D}
    - \sum_{\substack{\alpha \in \mathcal{A} \\ D_\alpha \ne E }}D_\alpha . E,
  \]
  and this number is nonnegative if and only if~\ref{enum:cond--1}
  and~\ref{enum:cond-no-three} hold for $E$.
\end{proof}

\begin{remark}\label{rmk:A-singularities}
  If $\rho\colon \tS\to S$ is the minimal desingularization of a singular del
  Pezzo surface, then
  Lemma~\ref{lem:nef-criterion} shows: If one of the
  $(-2)$-curves above a singularity $Q \in S$ is in $D$, then by
  Lemma~\ref{lem:nef-criterion}~\ref{enum:cond--2} all curves above this
  singularity must be in $D$. Similarly, if a $(-1)$-curve whose image in $S$
  contains a singularity $Q$ is in $D$, then all $(-2)$-curves above $Q$ must
  be in $D$. By Lemma~\ref{lem:nef-criterion}~\ref{enum:cond-no-three}, $Q$
  must be an $\Ab$-singularity in both cases.
\end{remark}

The surface $\tS$ can be described by a sequence of $r=9-\deg \tS$ blow-ups
\begin{equation*}
  \tS = \tS^{(r)} \xrightarrow{\pi_r} \tS^{(r-1)}\to \dots \to \tS^{(1)}
  \xrightarrow{\pi_1} \tS^{(0)} = \PP^2,
\end{equation*}
where $\pi_i$ is the blow-up in a point $p_i$ that
does not lie on a $(-2)$-curve on $\tS^{(i-1)}$.  Let $\pi\colon \tS \to \PP^2$ be their
composition.  Let $\ell_0=\pi^*\ell$, where $\ell$ is the class of a line on
$\PP^2$, and for $1\le i\le r$, let $\ell_i=(\pi_{i+1}\cdots\pi_{r})^{*}E^{(i)}$,
where $E^{(i)}$ is the exceptional divisor of the $i$th blow-up $\pi_i$.  Then
the Picard group of $\tS$ is freely generated by the classes
$\ell_0,\dots,\ell_r$. The intersection form is given by $\ell_i.\ell_j=0$ for $i\ne j$,
$\ell_0^2=1$, and $\ell_i^2=-1$ for $i\ge 1$. Let $P$ be the image of an
exceptional divisor of one of the the blow-ups in $\PP^2$, and $n_P$ be the
number of exceptional curves mapped to $P$. Then these negative curves form a chain,
the first $n_P-1$ of which are $(-2)$-curves whose classes
have the form $\ell_{i_1} - \ell_{i_2}$, \dots, $\ell_{i_{s-1}}-\ell_{i_s}$
followed by a $(-1)$-curve whose class has the form $\ell_{i_s}$.  The
anticanonical class is $3\ell_0 - \ell_1 - \cdots - \ell_r$, and we fix an
anticanonical divisor $-K$.  Denote by $[F]$ the class of a divisor or line
bundle $F$ in the Picard group. For $L_1,L_2\in \Pic (\tS)_\RR$, we write
$L_1 \le L_2$ if their difference $L_2 - L_1$ is in the effective cone.

\begin{lemma}\label{lem:aux-big-nef}
  Let $L\in\Pic(\tS)$.
  \begin{enumerate}[label = (\roman*)]
    \item\label{enum:aux-sum-lj} If $L\le \sum_{1\le j \le r} a_j \ell_j$ for some  $a_1,\dots,a_r\in\ZZ$, then $L$ is not big.
    \item\label{enum:aux-l0-li} If $L\le \ell_0 - \ell_i$ for some $i \ge 1$, then $L$ is not nef or not big.
  \end{enumerate}
\end{lemma}
\begin{proof}
  For the first statement, we just have to note that $-\eps\ell_0 + \sum_{1\le j \le k} a_j \ell_j$ is not effective for any $\eps>0$.
  Turning to the second statement, assume for contradiction that $L$ is big and nef. Note that $\ell_0 - \ell_i$ has nonnegative intersection with all $(-1)$-curves. 

  If $\ell_0-\ell_i$ has (strictly) negative intersection with a $(-2)$-curve $E$, then this curve needs to have class $[E]=\ell_i - \ell_j$ for some $j\ne i$~(cf.~\cite[Theorem 25.5.3]{MR833513}). Writing $\ell_0 - \ell_i = L + [F]$ with effective $F$, we get $F.E<0$, so $E \subset F$, and $L \le \ell_0 - \ell_i - [E] = \ell_0 - \ell_j$. The only negative curves that could have negative intersection with $\ell_0-\ell_j$ have class $\ell_j - \ell_k$. As curves of classes $\ell_i - \ell_j$, $\ell_j - \ell_k$, etc., are contracted to a single point by $\pi$, we can eventually find an $\ell_{i'}$ with $L\le \ell_0 - \ell_{i'}$ and such that $\ell_0-\ell_{i'}$ has nonnegative intersection with all negative curves.
  Then $\ell_0 -\ell_{i'}$ is nef. But $(\ell_0-\ell_{i'})^2=0$, whence it cannot be big.
\end{proof}

\begin{prop}\label{prop:big-necessary}
  Assume that $\deg \tS\le 4$. If $\omega_{\tS}(D)^\vee$ is big and nef, then $D$ is contained in the union of all negative curves.
\end{prop}
\begin{proof}
  Assume for contradiction that $D$ contains a nonnegative curve $C$, but that
  $-K-D$ is big and nef. In particular, $-K-D^\prime$ is big for all $D^\prime
  \subset D$. Since $C$ is nonnegative, it is the strict transform of a curve
  $C_0$ on $\PP^2$. Then
  \begin{equation*}
    [C]=d \ell_0 - \sum_{i=1}^r a_i \ell_i,
  \end{equation*}
  where $d=\deg C_0$ and $a_i=C.\ell_i$.
  
  We first reduce to the case of $C_0$ being a line. If $d\ge 3$, then $[-K-C] \le  \sum_{1\le i \le r} a_i \ell_i$, which is not big by Lemma~\ref{lem:aux-big-nef}~\ref{enum:aux-sum-lj}.  If $C_0$ is a nondegenerate conic, then $a_1,\dots,a_r\le 1$, since $C_0$ has multiplicity $\le 1$ in all images of the exceptional divisors. Moreover, since $C^2 \ge 0$, at most four of the $a_i$ are nonzero. It follows that $[-K-C]\le \ell_0 - \ell_j$, so $-K-D$ is not big or not nef by Lemma~\ref{lem:aux-big-nef}~\ref{enum:aux-l0-li}.

  So let $C_0$ be a line. As the self-intersection of $C$ is nonnegative, $[C]=\ell_0 - \ell_j$ for some $j$ or $[C]=\ell_0$. In the first case, $C_0$ contains the center $P=\pi_1 \cdots \pi_j(p_j)$ of a blow-up. If $n_P >1$, then $\pi^{-1}(P)$ contains $(-2)$-curves.
  Appealing to Lemma~\ref{lem:nef-criterion}~\ref{enum:cond--2}, the first $(-2)$-curve must be contained in $D$, as must the remaining $(-2)$-curves by repeated applications. Let $E_0$ be the sum of these $(-2)$-curves.  Then $C'=C+E_0\subset D$ is of class $[C']=\ell_0 -\ell_{j'}$, where $\ell_{j'}$ is the class of the final $(-1)$-curve in the chain. If $n_P=1$ or $[C]=\ell_0$, set $E_0=0$ and $C'=C$; in the first case, set $j'=j$; in the latter case, fix an arbitrary $j'$ and note that $[C]\le \ell_0 - \ell_{j'}$. Then $C'$ satisfies the conditions in Lemma~\ref{lem:nef-criterion} for all negative curves in the preimage of $P$ by this construction, and it does the same for all other curves contracted by $\pi$ as it does not meet them.
  For what remains, we distinguish three cases.

  Case 1. The curve $C$ does not meet any of the remaining $(-2)$-curves, and $C.E\le 1$ for all remaining $(-1)$-curves. Then $(-K-C')$ is nef by Lemma~\ref{lem:nef-criterion}. But $(-K-C')^2\le 4 - (r - 1) \le 0$, so it cannot be big.

  Case 2. The curve $C$ meets one of the remaining $(-2)$-curves $E$. Then
  $E\subset D$ by Lemma~\ref{lem:nef-criterion}~\ref{enum:cond--2}.  Since $E$
  is the strict transform of a curve in $\PP^2$, its class satisfies
  $[E]= l_0 -l_{i_1} - l_{i_2} - l_{i_3}$ for some pairwise different
  $i_1, i_2, i_3$ or $[E] \ge 2 \ell_0 + \sum a_i\ell_i$ for some
  $a_i\in\ZZ$. In the first case, $[-K-C-E]\le \ell_0 - \ell_k$ for
  $k\ne i_i,i_2,i_3,j$, and in the second case,
  $[-K-C-E] \le \sum_{1 \le i \le r} a_i\ell_i$. In both cases, $-K-D$ is not
  big or not nef by Lemma~\ref{lem:aux-big-nef}~\ref{enum:aux-l0-li}
  or~\ref{enum:aux-sum-lj}, respectively.

  Case 3. The curve $C$ meets a $(-1)$-curve $E$ with $C.E\ge 2$.  By
  Lemma~\ref{lem:nef-criterion}~\ref{enum:cond--1}, $E\subset D$. As $E$ is
  the strict transform of a curve on $\PP^2$, its class verifies $[E]\ge [F]$
  for a $(-2)$-class $[F]$ of the same shape as in the previous case; hence,
  $-K-D$ is not big or not nef.
\end{proof}

\begin{remark}
  The assumption $\deg \tS \le 4$ in Proposition~\ref{prop:big-necessary} is
  necessary: Let $\tS$ be a smooth del Pezzo surface of degree at least $5$
  that is a blow-up of $\PP^2$ in at most $4$ points in general position. Then
  the strict transform $D$ of a line that meets precisely one of these points
  is an example of a nonnegative curve such that $\omega_{\tS}(D)^\vee$ is big
  and nef.
\end{remark}

\begin{theorem}\label{thm:classification}
  Let $\tS$ be a weak del Pezzo surface of degree $d \le 4$. Precisely the following choices of a reduced effective divisor $D$ make $(\tS,D)$ a weak del Pezzo pair.
  \begin{enumerate}[label = (\roman*)]
    \item The divisor $D$ can be zero.
    \item If $3 \le d \le 4$, then $D$ can consist of all $(-2)$-curves corresponding to one $\Ab$-singularity.
    \item If $d= 4$, then $D$ can consist of a $(-1)$-curve and all $(-2)$-curves corresponding to all singularities on its image in the anticanonical model, provided that those singularities are $\Ab$-singularities and all curves in $D$ form a chain.
  \end{enumerate}
\end{theorem}
\begin{proof}
  Let $D$ be a reduced effective divisor such that $-K-D$ is big and nef. By Proposition~\ref{prop:big-necessary},  $D=\sum E_i$ has to be supported on negative curves.
  Consider the complete subgraph $G$ of the Dynkin diagram on the vertices corresponding to components of $D$.  By Lemma~\ref{lem:nef-criterion}~\ref{enum:cond-no-three}, each of its connected components is a path or a cycle.
  Let $N_1$ be the number of $(-1)$-curves in $D$, and $N_2$ be the number of $(-2)$-curves. Then $v=N_1+N_2$ is the number of vertices of $G$, and denote by $e$ its number of edges.

  The self-intersection of the log-anticanonical divisor is
  \[
    (-K-D)^2 = K^2 + \sum_i E_i^2 + 2 \sum_i E_i.K + 2\sum_{i<j} E_i.E_j.  
  \]
  As $-K.E$ is zero for $(-2)$-curves and $1$ for $(-1)$-curves, we get
  \begin{equation}\label{eq:self-intersection}
    (-K-D)^2 =  d + 2(e-v) - N_1.
  \end{equation}
  Since $-K-D$ is big and nef, this self-intersection must be positive.

  If $G$ is connected and not a cycle, then $e=v-1$, so $d -2-N_1 >0$.
  In case $d=4$, this leaves us with $N_1 \le 1$, in case $d=3$ with $N_1 = 0$, and in case $d\le 2$ with an immediate contradiction. In each case, the resulting divisors satisfy the asserted description using Remark~\ref{rmk:A-singularities} and that the graph is a path. 

  It remains to prove that $G$ has to be connected and not a cycle. If $G$ is not connected and does not contain a cycle, then $(-K-D)^2 = d-4-N_1 \le 0$, so $-K-D$ cannot be big, leaving only the case of graphs $G$ containing a cycle, in which case $d-N_1>0$ for $-K-D$ to be big.

  For $N_1=0$, we note that only Dynkin diagrams of type $\Ab$, $\mathbf D$, and $\mathbf E$ appear as intersection graphs of $(-2)$-curves, and these do not contain double edges, nor more general cycles.

  If $N_1=1$, then $d\ge 2$. The sum $E_2$ of the $(-2)$-curves in $D$ forms a
  $(-2)$-class, since $E_2^2 = -2N_2 + 2(s-1)=-2$ and $-K.E_2=0$.  As the Weyl
  group acts transitively on $(-1)$-curves and leaves the intersection pairing
  invariant, we can assume that $[E_1]=\ell_1$. Now $\ell_1.[E_2]\ge 2$, and
  so the $(-2)$-class needs to have the form
  $3\ell_0 - 2\ell_1 - \ell_2 - \cdots - \ell_8$. But such a class does not exist if
  $d\ge 2$ (cf.~\cite[Theorem 25.5.3]{MR833513}).

  If $N_1 = 2$, then $d\ge 3$. In this case, the anticanonical model contracts $(-2)$-curves and maps $(-1)$-curves to lines. The resulting two lines then need to intersect with multiplicity $2$, an impossibility.

  Finally, if $N_1=3$, then $d = 4$. In this case, the anticanonical model $\phi\colon \tS\to S = Q_1 \cap Q_2 \subset \PP^4$ is a (possibly singular) intersection of two quadrics, still contracting all $(-2)$-curves and mapping all $(-1)$-curves to lines. The resulting three lines need to intersect pairwise. If they were contained in a plane $P$, this plane would intersect $Q_1$ in three lines, an impossibility. So the three lines intersect in a point $Q$. The tangent space at $Q$ needs to contain each plane containing two of these lines, whence $Q$ is singular.
  Then  $\tS\to S$ factors through the blow-up $Y$ of $\PP^4$ in $Q$. The strict transforms of the lines do not intersect on $Y$, and thus the $(-1)$-curves on $\tS$ do not intersect. It follows that each of them intersects a $(-2)$-curve above $Q$. Hence, the $(-2)$-curves are contained in $D$, and at least one of them needs to intersect three other negative curves in $D$. Now, Lemma~\ref{lem:nef-criterion}~\ref{enum:cond-no-three} implies that $-K-D$ cannot be nef.
    
  Conversely, if $D$ is one of the divisors in the statement, then it is nef by Lemma~\ref{lem:nef-criterion}, and its self-intersection is positive by~\eqref{eq:self-intersection}; hence, it is also big.
\end{proof}

\section{Passage to a universal torsor}\label{sec:torsor}

As in the introduction, let $S \subset \PP^4_\QQ$ be the singular quartic del
Pezzo surface defined by the equations~\eqref{eq:def_eq}. 
By~\cite{MR3180592,MR3939263}
(but using the notation and numbering of~\cite[Section~8]{MR2520770}), a Cox
ring of its minimal desingularization
$\tS$ is
\begin{equation}\label{eq:cox}
  R=\QQ[\e_1,\dots,\e_9]/(\e_1\e_9+\e_2\e_8+\e_4\e_5^3\e_6^2\e_7)
\end{equation}
with grading 
\begin{equation}\label{eq:Pic_degrees}
  \begin{aligned}
    \deg\e_1&=\ell_5,\quad \deg\e_2 =\ell_4,\quad \deg\e_3=\ell_0-\ell_1-\ell_4-\ell_5,\\
    \deg\e_4&=\ell_1-\ell_2,\quad \deg\e_5=\ell_3,\quad \deg\e_6=\ell_2-\ell_3,\\
    \deg\e_7&=\ell_0-\ell_1-\ell_2-\ell_3,\quad \deg\e_8=\ell_0-\ell_4,\quad \deg\e_9=\ell_0-\ell_5
  \end{aligned}
\end{equation}
for a certain basis $\ell_0,\dots,\ell_5$ of $\Pic\tS$. See
Figure~\ref{fig:dynkin-diagram} for the dual graph of the divisors $E_i$
corresponding to $\e_i$.
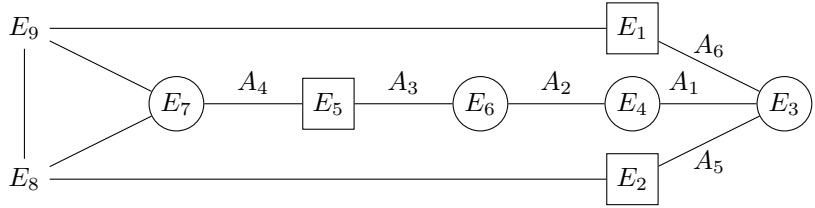
\begin{figure}[ht]
  \begin{center}
    \begin{tikzpicture}[square/.style={regular polygon,regular polygon sides=4}]
      \node (E8) at (0,-1) {$E_8$};
      \node (E9) at (0, 1) {$E_9$};
      \node[circle, draw, inner sep=.5ex] (E7) at (2, 0) {$E_7$};
      \node[square, draw, inner sep=.2ex] (E5) at (4, 0) {$E_5$};
      \node[circle, draw, inner sep=.5ex] (E6) at (6, 0) {$E_6$};
      \node[circle, draw, inner sep=.5ex] (E4) at (8, 0) {$E_4$};
      \node[circle, draw, inner sep=.5ex] (E3) at (10, 0) {$E_3$};
      \node[square, draw, inner sep=.2ex] (E1) at (8, 1) {$E_1$};
      \node[square, draw, inner sep=.2ex] (E2) at (8,-1) {$E_2$};

      \draw (E8) -- (E9);
      \draw (E8) -- (E7);
      \draw (E9) -- (E7);
      \draw (E9) -- (E1);
      \draw (E8) -- (E2);
      \draw (E3) -- (E1) node [midway, above] {$A_6$};
      \draw (E3) -- (E2) node [midway, below] {$A_5$};
      \draw (E7) -- (E5) node [midway, above] {$A_4$};
      \draw (E5) -- (E6) node [midway, above] {$A_3$};
      \draw (E6) -- (E4) node [midway, above] {$A_2$};
      \draw (E4) -- (E3) node [near start, above] {$A_1$};
    \end{tikzpicture}
    \caption{Configuration of the divisors $E_i$ and the faces $A_i$ of the
      Clemens complexes.
      The $(-1)$-curves are represented by squares and the $(-2)$-curves by circles.}\label{fig:dynkin-diagram}
  \end{center}
\end{figure}

The minimal desingularization $\tS$ can be described as a certain sequence of five
iterated blow-ups of $\PP^2_\QQ$ in rational points~\cite{MR2520770}: first blow up three points $P_1$, $P_2$, $P_4$ on a line $l_3$, resulting in exceptional curves $E_1$, $E_2$, and $E_4'$; then blow up the intersection of $E_4'$ with a line $l_7$, resulting in an exceptional curve $E_6'$; then blow up the intersection of $E_6'$ with the strict transform of $l_7$, resulting in an exceptional curve $E_5$. With this description, $E_3$ is the strict transform of $l_3$, $E_4$ that of $E_4'$, $E_6$ that of $E_6'$, $E_7$ that of $l_7$, $E_8$ that of a general line through $P_2$, and $E_9$ that of a line through $P_1$ such that $E_7,E_8,E_9$ meet in one point, recovering the above grading using a basis as before Lemma~\ref{lem:aux-big-nef}.

With a point of view coming from $S$, the divisor $D_1=E_7$ is the $(-2)$-curve on $\tS$ above the singularity
$Q_1$ on $S$, the divisor $D_2=E_3+E_4+E_6$ is the sum of the $(-2)$-curves
above $Q_2$, the divisor $D_3=D_1+D_2=E_3+E_4+E_6+E_7$ is the sum of all
$(-2)$-curves, and $E_5,E_2,E_1$ are the $(-1)$-curves that are the strict
transforms of the three lines $L_1,L_2,L_3$ on $S$ as in~\eqref{eq:lines}, respectively,
while $E_8$ and $E_9$ correspond to the two further generators of the Cox ring.
The divisors
\begin{equation*}
  D_4=E_3+\cdots+E_7,\quad D_5=E_2+E_3+E_4+E_6,\quad \text{and} \quad D_6=E_1+E_3+E_4+E_6
\end{equation*}
lie above the lines $L_1$, $L_2$, and $L_3$, respectively.
Since $V \subset S$ is the complement of the lines, which contain the
singularities, its preimage $\tV \subset \tS$ is the complement of the
negative curves $E_1,\dots,E_7$.  

The irrelevant ideal of $R$ is $I_\irr=\prod(\e_i,\e_j)$, where the product
runs over all pairs $i < j$ such that there is no edge between $E_i$ and $E_j$
in Figure~\ref{fig:dynkin-diagram}.  The sections
\begin{equation}\label{eq:antican_sections}
  L_0=\{\e_2\e_3\e_4\e_5\e_6\e_7\e_8,\
  \e_1^2\e_2^2\e_3^3\e_4^2\e_6,\
  \e_1\e_2\e_3^2\e_4^2\e_5^2\e_6^2\e_7,\
  \e_3\e_4^2\e_5^4\e_6^3\e_7^2,\
  \e_7\e_8\e_9
  \}
\end{equation}
have anticanonical degree and define the morphism $\rho: \tS\to S$.

As in~\cite[Proposition~4.1(i)]{MR3552013}, let $\tSs$ be the integral model
defined by the corresponding sequence of blow-ups of $\PP^2_\ZZ$, and recall
$\tUs_i=\Ss-\overline{D_i}$. Consider the open subscheme $\Ys$ of the spectrum
of
\begin{equation*}
R_\ZZ=\ZZ[\e_1,\dots,\e_9]/(\e_1\e_9+\e_2\e_8+\e_4\e_5^3\e_6^2\e_7)
\end{equation*}
defined as the complement of $\VV(I_\irr\cap
R_\ZZ)$. By~\cite[Proposition~4.1(ii)]{MR3552013}, $\Ys$ is a
$\GGmZ^6$-torsor over $\tSs$ via a morphism $\pi:\Ys\to\tSs$; here,
the action of $\GGmZ^6$ on $\Ys$ is given by the degrees of the
coordinates $(\e_1,\dots,\e_9)$ in $\Pic \tS \cong \ZZ^6$ in
\eqref{eq:Pic_degrees}; see \cite[Construction~3.1]{MR3552013} for details.  This torsor defines an explicit
parametrization of integral points by lattice points:

\begin{lemma}\label{lem:torsor}
  Let $i\in\{1,\dots,6\}$, and $\Ys_i=\pi^{-1}(\tUs_i)\subset \Ys$. Then $\pi\colon \Ys_i\to \tUs_i$ is a $\GGmZ^6$-torsor. This morphism induces a $2^6$-to-$1$-correspondence
  \[
    \Ys_i(\ZZ)\cap \pi^{-1}(\tV)(\QQ) \to \tUs_i(\ZZ)\cap \tV(\QQ),
  \]
  and we have
  \[
    \Ys_i(\ZZ)\cap \pi^{-1}(\tV)(\QQ) = 
    \where{
      \ee \in \ZZ^9
    }{
      \text{\eqref{eq:torsor},~\eqref{eq:surface-gcd},~\eqref{eq:pm1} hold},\ \e_1\cdots\e_7 \ne 0,\\
     },
  \]
  where
  \begin{align}
    &\e_1\e_9+\e_2\e_8+\e_4\e_5^3\e_6^2\e_7=0,\label{eq:torsor}\\
    &\gcd(\e_i,\e_j)=1 \qquad \text{if $E_i$ and $E_j$ do not share an edge in
    Figure~\ref{fig:dynkin-diagram}}, and\label{eq:surface-gcd}\\
    &|\e_j| = 1 \qquad \text{ if } E_j \subset \abs{D_i}.\label{eq:pm1}
  \end{align}
\end{lemma}
\begin{proof}
  Since $\pi$ is a $\GGmZ^6$-torsor, so are its restrictions to the open
  subschemes $\tUs_i$.  
  Integral points $\ee\in\Spec R_\ZZ$ are lattice points
  $(\e_1,\dots,\e_9)\in\ZZ^9$ satisfying the equation in the Cox ring. Such a
  point is integral on the complement of
  $\VV(I_\irr)=\bigcup \VV(\e_i,\e_j)$---the union running over all $i,j$ which
  do not share an edge in Figure~\ref{fig:dynkin-diagram}---if it does not
  reduce to any of the $\VV(\e_i,\e_j)$ for any prime, that is,
  if~\eqref{eq:surface-gcd} holds.
  
  Integral points on $\Ys_1\subset \Ys$ are precisely those which do not
  reduce to $\pi^{-1}(E_7)=\VV(\e_7)$ at any place; that is, they are those points satisfying
  $\e_7\in\{\pm 1\}$. Analogously, integral points on $\Ys_2$ are those
  satisfying $\e_3,\e_4,\e_6\in\{\pm 1\}$, integral points on $\Ys_3$ are
  those satisfying $\e_3,\e_4,\e_6,\e_7\in\{\pm 1\}$, and similarly for
  $\Ys_4, \Ys_5, \Ys_6$.  The preimage of $\tV$
  in the universal torsor is the complement of $\e_1\cdots\e_7=0$.
\end{proof}

We now turn to studying the log-anticanonical bundles and the height functions
associated with them. Recall that the case
$D_6$ can be reduced to $D_5$ by symmetry as in~\eqref{eq:symmetry}.

\begin{lemma}\label{lem:big-and-nef}
  The only nonzero reduced effective divisors $D\subset \tS$ such that
  $\omega_{\tS}(D)^\vee$ is big and nef are $D_i$ for $i \in \{1,2,4,5,6\}$.
  Consider the sets 
  \begin{align*}
    &M_1= \{\e_2\e_3\e_4\e_5\e_6\e_8,\e_1\e_2\e_3^2\e_4^2\e_5^2\e_6^2,\e_3\e_4^2\e_5^4\e_6^3\e_7,\e_8\e_9\},\\
    & M_2 = \{\e_2\e_5\e_7\e_8,\e_1^2\e_2^2\e_3^2\e_4,
      \e_4\e_5^4\e_6^2\e_7^2\},\\
    &M_4=\{\e_2\e_8, \e_1\e_2\e_3\e_4\e_5\e_6, \e_4\e_5^3\e_6^2\e_7\}, \quad \text{and}\\ 
    &M_5=\{\e_5\e_7\e_8,\e_1^2\e_2\e_3^2\e_4,\e_1\e_3\e_4\e_5^2\e_6\e_7\}
  \end{align*}
  of monomials in the Cox ring $R$ of degree $\omega_{\tS}(D_i)^\vee$ for $i=1,2,4,5$,
  respectively. For $\ee\in\ZZ^9$ satisfying~\eqref{eq:surface-gcd}, none of these sets can vanish simultaneously modulo a prime $p$.
  The respective log-anticanonical bundles are base point free.

  The log-anticanonical bundle $\omega_{\tS}(D_3)^\vee$ is big, but not
  nef, whence not base point free. It has a representation
  $\omega_{\tS}(D_3)^\vee\cong \Ls_1 \otimes \Ls_2^\vee$ as a quotient of the nef
  bundles $\Ls_1$ and $\Ls_2$ whose sections are elements of degree
  $4\ell_0-\ell_1-\ell_2$ and $3\ell_0-\ell_1-\ell_2-\ell_3$ in the Cox ring,
  respectively.
  Consider the sets of Cox ring elements
  \begin{align*}
      L_2  &= 
         \e_1\e_2L_0 
      \cup \{\e_4^2 \e_5^6 \e_6^4 \e_7^3\} && \text{of degree $\Ls_2$,}\\
      L_3 &= \{\e_2\e_5\e_8, \e_1\e_2\e_3\e_4\e_5^2\e_6, \e_4 \e_5^4 \e_6^2
            \e_7\} && \text{of degree $\omega_{\tS}(D_3)^\vee$, and}\\
      L_1 &= L_2 L_3 \cup
        \{\e_1^3 \e_2^3 \e_3^2 \e_4 \e_8 \e_9\} && \text{of degree $\Ls_1$.}
  \end{align*}
  Then neither $L_1$ nor $L_2$ can vanish simultaneously modulo a prime $p$.
\end{lemma}
\begin{proof}
  The first statement is a special case of Theorem~\ref{thm:classification}.
  
  For the first set, assume that $p\mid \eta_8\eta_9$ for a prime $p$. Then $p\nmid \e_3\cdots\e_6$,
  since the corresponding divisors $E_3,\dots,E_6$ share an edge
  with neither $E_8$ nor $E_9$ in Figure~\ref{fig:dynkin-diagram},
  while at most one of
  $\e_1,\e_2,\e_7$ can be divisible by $p$. Hence, the second or third section is not divisible by $p$.

  For the second set, assume that $p\mid \e_2\e_5\e_7\e_8$. If $p\mid \e_5\e_7$, then $p\nmid  \e_1^2\e_2^2\e_3^2\e_4$;
  if $p\mid \e_2\e_8$ and $p\nmid\e_7$, then $p\nmid \e_4\e_5^3\e_6^2\e_7$.
  Regarding the case $i=4$, if $p\mid \e_2\e_8$, then $p\nmid \e_4\e_5^3\e_6^2\e_7$.
  For  $i=5$, assume $p\mid \e_5\e_7\e_8$. If $p\mid \e_5\e_7$, then $p\nmid \e_1^2\e_2\e_3^2\e_4$;
  if $p\mid \e_8$ and $p\nmid \e_7$, then $p\nmid \e_1\e_3\e_4\e_5^2\e_6\e_7$. 

  Turning to $L_2$, we know that the anticanonical sections
  in~\eqref{eq:antican_sections} cannot be divisible by $p$ simultaneously,
  so for all sections in $L_2$ to
  be be divisible by $p$ simultaneously, $p\mid \e_1\e_2$. But then
  $p\nmid \e_4^2 \e_5^6 \e_6^4 \e_7^3$.
  
  Lastly, assume that the
  monomials in $L_1$ are divisible by $p$, so that $p$ in particular divides all monomials in $L_2L_3$.
  Since $p$ cannot divide all monomials in $L_2$, it has to divide all monomials in $L_3$.
  If follows that $p\mid \e_2\e_5\e_8$.  If $p\mid\e_2$, then the last monomial in $L_3$ cannot be zero modulo $p$.
  If $p\mid\e_8$, then
  $p$ can divide only one of $\e_2$ or $\e_7$, but none of the remaining variables in
  the latter two sections of $L_3$, so one of those two sections
  is nonzero modulo $p$.  So $p\mid \e_5$. Now $p$ can divide at most one of $\e_6$ and $\e_7$,
  but none of the remaining variables, and so the last section of
  $L_1$ is nonzero modulo $p$.  A contradiction.

  By the same arguments (replacing vanishing modulo $p$ by vanishing over $\Qbar$), the log-anticanonical bundles
  $\omega_{\tS}(D_i)^\vee$ for $i=1,2,4,5$ and the bundles $\Ls_1$ and
  $\Ls_2$ are base point free, whence nef. On the other hand,
  $\omega_{\tS}(D_3)^\vee$ is not nef, since its intersection number with $E_5$
  is $-1$.
\end{proof}

\begin{lemma}
  For $i \in \{1,\dots,5\}$, the morphism $\rho\colon \tS\to S$ induces bijections
  \begin{equation*}
    \tUs_i(\ZZ) \cap \tV(\QQ) \to \Us_i(\ZZ)\cap V(\QQ).
  \end{equation*}
\end{lemma}
\begin{proof}
  Consider the morphism $f\colon \Ys\to \PP^4_\ZZ$ defined by $\ee\mapsto(s_0(\ee):\dots:s_4(\ee))$
  with the anticanonical sections $s_j$ in~\eqref{eq:antican_sections}. We only have to show that~\eqref{eq:surface-gcd} holds for $\ee$ if and only if the corresponding gcd-condition in~\eqref{eq:gcd} resp.~\eqref{eq:gcd-lines} holds for $f(\ee)$.
  To this end, we note that $\e_1\e_2 \e_3\e_4\e_5^2\e_6\e_7$ is in the radical of the ideal generated by $M_2$ as in Lemma~\ref{lem:big-and-nef},
  so~\eqref{eq:gcd} can be rewritten as
  \[
    \abs{\e_7}\gcd\{m(\ee) \mid m \in M_1\} = 1
    \quad \text{and}\quad
    \abs{\e_3\e_4\e_6}
    \gcd\{m(\ee) \mid m \in M_2\} = 1.
  \] 
  By Lemma~\ref{lem:big-and-nef}, these gcds are one, and so the claim follows for $i\in\{1,2\}$. Since $\Us_3$ and $\tUs_3$ are the intersections of the respective open subschemes with in the first two cases, the assertion follows for $i=3$.
  The cases $i=4$ and $i=5$ can be proved using an analogous reformulation of the first two conditions in~\eqref{eq:gcd-lines}.
\end{proof}

These sets of sections define adelic metrics on line bundles isomorphic
to $\omega_{\tS}(D_i)^\vee$ for $i \in \{1,2,4,5\}$ and the line bundles $\Ls_1$ and
$\Ls_2$, and the latter two metrics induce one on
$\Ls_1\otimes \Ls_2^\vee\cong \omega_{\tS}(D_3)^\vee$. The metrics on the bundles
isomorphic to the log-anticanonical bundles then induce log-anticanonical
height functions $\tH_i$ for $i \in \{1,\dots,5\}$.

\begin{lemma}\label{lem:heights}
  For $\ee = (\e_1,\dots,\e_9) \in \RR^9$ satisfying~\eqref{eq:torsor}
  and~\eqref{eq:pm1}, let
  \begin{align*}
    \Hs_i(\ee) = 
    \begin{cases}
      \max\{|\e_2\e_3\e_4\e_5\e_6\e_8|,|\e_1\e_2\e_3^2\e_4^2\e_5^2\e_6^2|,
      |\e_3\e_4^2\e_5^4\e_6^3|,|\e_8\e_9|\},
      & i=1;\\
    \max\{|\e_2\e_5\e_7\e_8|,|\e_1^2\e_2^2|,|\e_5^4\e_7^2|\}, & i=2;\\
      \max\{|\e_2\e_5\e_8|,|\e_1\e_2\e_5^2|,|\e_5^4|,
      \min\{\abs{\e_1^2\e_2^2},|\e_8\e_9|
      \}\}, & i=3;\\
    \max\{|\e_2\e_8|,|\e_1\e_2|,1\}, & i=4;\\
    \max\{|\e_1^2|,|\e_1\e_5^2\e_7|,|\e_5\e_7\e_8|\}, & i=5.
    \end{cases}
  \end{align*}
  For $\ee \in \Ys_i(\ZZ) \cap \pi^{-1}(\tV)(\QQ)$, we have
  $\Hs_i(\ee) = \tH_i(\pi(\ee)) = H_i(\rho(\pi(\ee)))$, where $H_i$ is the
  height in~\eqref{eq:height} and $\tH_i$ is the log-anticanonical height on
  $\tS(\QQ)$ induced by the sections in
  Lemma~\ref{lem:big-and-nef}.
\end{lemma}
\begin{proof}  
  For $i\in\{1,2,4,5\}$, the metrics induce height functions $\tH_i(x)=H_{\PP^{N_i}}(f_i(x))$
  verifying $f_i(\pi(\ee)) = (m_0(\ee): \dots : m_{N_i}(\ee))$ for the sections $m_0,\dots,m_{N_i}\in M_i$
  constructed in Lemma~\ref{lem:big-and-nef},
  so
  \[
    H_i(\pi(\ee)) = \prod_v \max_{m\in M_i}\{\abs{m(\ee)}_v\}.
  \]
  By the same lemma, the $p$-adic contributions to 
  this product are $1$, whence $\tH_i(\pi(\ee)) = \Hs_i(\ee)$.

  To check
  that these height functions coincide with the ones defined in the introduction, we note that,
  for example, for a point in $\Ys_1(\ZZ)$ we have $\e_7\in \{\pm 1\}$, and
  thus $|\e_2\e_3\e_4\e_5\e_6\e_8|=|\e_2\e_3\e_4\e_5\e_6\e_7\e_8|=|x_0|$. We
  get analogous identities for the other coordinates and cases.  There is no section
  corresponding to $x_2$ in the second height function, but, for integral
  points on $\Ys_2$, we have
  $\abs{x_2}=\sqrt{\abs{\e_1^2\e_2^2}\abs{\e_5^4\e_7^2}}=\sqrt{\abs{x_1x_3}}$;
  hence, it can never contribute to the maximum.
  
  The case $i=3$ is more complicated.
  The log-anticanonical height function $\tH_3$ induced by the metrics on
  $\Ls_1$ and $\Ls_2$ satisfies
  \begin{equation*}
    \tH_3(\pi(\ee))=\frac{\max_{s\in L_1} \abs{s(\ee)}}{\max_{s\in L_2} \abs{s(\ee)}}
    = \max\left\{ \max_{s \in L_3} \abs{s(\ee)},
    \frac{\abs{\e_1^3 \e_2^3 \e_3^2 \e_4 \e_8 \e_9}}{\max_{s \in L_2} \abs{s(\ee)}}
    \right\}
  \end{equation*}
  by an analogous argument as in the previous cases. Now note that for
  $\ee \in \Ys_3(\ZZ) \cap \pi^{-1}(\tV)(\QQ)$,
  we can simplify this to
  \begin{equation}\label{eq:new-height-3}
    \widetilde H_3(\pi(\ee))=
    \max\{\abs{\e_2\e_5\e_8}, \abs{\e_1\e_2\e_5^2}, \abs{\e_5^4},
    \min\{\abs{\e_1^2\e_2^2}, \abs{\e_8\e_9}\}\}.
  \end{equation}
  Indeed, $\e_3,\e_4,\e_6,\e_7$ have absolute value $1$ and the remaining
  variables absolute value at least one. Then
  \[
    \frac{\abs{\e_1^3 \e_2^3 \e_3^2 \e_4 \e_8 \e_9}}{\max_{s \in L_2} \abs{s(\ee)}} \le \min\{\abs{\e_1^2\e_2^2}, \abs{\e_8\e_9}\}
  \]
  follows from $\e_1^3\e_2^3\e_3^4\e_4^2\e_6, \e_1\e_2\e_7\e_8\e_9\in L_2$.  Now
  assume that $\e_1^2\e_2^2$ is larger than the other three terms in the
  maximum in~\eqref{eq:new-height-3}; then the maximum over $L_2$ can only be
  attained by $\e_1^3\e_2^3\e_3^4\e_4^2\e_6$ or $\e_1\e_2\e_7\e_8\e_9$, and the
  inverse inequality follows.

  Finally, we note
  that $\abs{\e_3}=\abs{\e_4}=\abs{\e_6}=\abs{\e_7}=1$ implies
  $\abs{\e_1^2\e_2^2}=\abs{x_1}$ and $\abs{\e_8\e_9}=\abs{x_4}$.
\end{proof}

\begin{lemma}\label{lem:param}
  For $i \in \{1,\dots,5\}$, we have 
  \begin{equation*}
    N_i(B) = \frac{1}{2^{6-\#D_i}}\#\bigwhere{
      (\e_1\dots,\e_9) \in \ZZ^{9}
    }{& \text{\eqref{eq:torsor},~\eqref{eq:surface-gcd} hold},\ \e_j = 1\text{ if }E_j \subset D_i,\\
      & \e_1\cdots\e_7 \ne 0,\ \Hs_i(\e_1,\dots,\e_9) \le B
    },
  \end{equation*}
  where $\#D_i$ denotes the number of irreducible components of $D_i$.
\end{lemma} 
\begin{proof}
  We combine Lemma~\ref{lem:torsor} and Lemma~\ref{lem:heights}.  The $\#D_i$
  coordinates $\e_j$ belonging to irreducible components $E_j\subset D_i$ satisfy
  $\abs{\e_j}=1$. By symmetry, we can further assume that $\e_j=1$, making the
  $2^6$-to-$1$-correspondence from Lemma~\ref{lem:torsor} a
  $2^{6-\#D_i}$-to-$1$-correspondence.
\end{proof}

\section{Counting}\label{sec:counting}

In our counting process, we treat $\e_9$ as a dependent variable using the
torsor equation from~\eqref{eq:cox}, which we regard as a congruence modulo the coefficient $\e_1$
of $\e_9$. First, we sum over $\e_8$, and then over the remaining
variables. Since this is similar to the case of rational points
in~\cite{MR2520770}, we shall be brief.

In this section, we use the notation
\begin{equation*}
  \eei = (\e_j)_{j \in J_i} =
  \begin{cases}
    (\e_1,\dots,\e_6), &i=1;\\
    (\e_1,\e_2,\e_5,\e_7), &i=2;\\
    (\e_1,\e_2,\e_5), &i=3;\\
    (\e_1,\e_2), &i=4;\\
    (\e_1,\e_5,\e_7), &i=5;
  \end{cases}
\end{equation*}
for $(7-\#D_i)$-uples indexed by 
\begin{equation}\label{eq:def_Ji}
J_i=\{j \in \{1,\dots,7\} \mid E_j \not\subset D_i\}.
\end{equation}
We write $\Hs_i(\eei,\e_8)$ for $\Hs_i(\e_1,\dots,\e_9)$ where $\e_j=1$
whenever $E_j \subset D_i$ and where $\e_9$ is expressed in terms of
$\e_1,\dots,\e_8$ using the torsor equation~\eqref{eq:torsor}, assuming
$\e_1\ne 0$.

\begin{lemma}\label{lem:first_summation}
  For $i \in \{1,\dots, 5\}$, we have
  \begin{equation*}
    N_i(B) = \frac{1}{2^{6-\#D_i}} \sum_{\eei \in \ZZ_{\ne 0}^{J_i}}
    \theta_1(\eei)V_{i,1}(\eei;B) + O(B \log B)
  \end{equation*}
  with
  \begin{equation*}
    V_{i,1}(\eei;B) = \int_{\Hs_i(\eei,\e_8) \le B}  \frac{\diff \e_8}{|\e_1|}
  \end{equation*}
  and
  \begin{equation*}
    \theta_1(\eei) = \prod_p \theta_{1,p}(I_p(\eei)),
  \end{equation*}
  where $I_p(\eei) = \{j \in J_i \mid p \mid \e_j\}$
  and
  \begin{equation*}
    \theta_{1,p}(I) =
    \begin{cases}
      1, & I = \emptyset, \{1\}, \{2\}, \{7\};\\
      1-\frac 1 p, & I = \{4\}, \{5\}, \{6\}, \{1,3\}, \{2,3\}, \{3,4\};
      \{4,6\}, \{5,6\}, \{5,7\};\\
      1- \frac 2 p, & I = \{3\};\\
      0, & \text{otherwise.}
    \end{cases}
  \end{equation*}
\end{lemma}
\begin{proof}
  The proof is as in~\cite[Lemma~8.4]{MR2520770}, with slightly
  different height functions and some $\e_i=1$, which leads to
  different error terms. In the first case, using the second height
  condition, the error term is
  \begin{equation*}
    \ll \sum_{\e_1, \dots, \e_6} 2^{\omega(\e_3)+\omega(\e_3\e_4\e_5\e_6)} \ll
    \sum_{\e_2, \dots, \e_6}
    \frac{2^{\omega(\e_3)+\omega(\e_3\e_4\e_5\e_6)}B}{|\e_2\e_3^2\e_4^2\e_5^2\e_6^2|}
    \ll B\log B.
  \end{equation*}
  In the second case, using the second and the third height condition, it is
  \begin{equation*}
    \ll \sum_{\e_1, \e_2, \e_5, \e_7} 2^{\omega(\e_5)} \ll
    \sum_{\e_1, \e_5}
    \frac{2^{\omega(\e_5)}B}{|\e_1\e_5^2|}
    \ll B\log B.
  \end{equation*}
  In the third case, using the second height condition, it is
  \begin{equation*}
    \ll \sum_{\e_1, \e_2, \e_5} 2^{\omega(\e_5)} \ll
    \sum_{\e_1, \e_5}
    \frac{2^{\omega(\e_5)}B}{|\e_1\e_5^2|}
    \ll B\log B.
  \end{equation*}
  The remaining cases are very similar.
\end{proof}

\begin{lemma}\label{lem:surface-result-of-count}
  For $i \in \{1,\dots, 5\}$, we have
  \begin{equation*}
    N_i(B) =\frac{1}{2^{6-\#D_i}}\left(\prod_p \omega_{i,p}\right) V_{i,0}(B) + O(B(\log B)^{b_i-2}\log \log B)
  \end{equation*}
  with
  \begin{equation*}
    V_{i,0}(B) =\int_{|\e_j| \ge 1 \ \forall j \in J_i}
    V_{i,1}(\eei;B) \diff \eei
  \end{equation*}
  and
  \begin{equation*}
    \omega_{i,p} =
    \begin{cases}
      \left(1-\frac 1 p\right)^{6-\#D_i}\left(1+\frac{6-\#D_i}
        p\right), & i \in \{1,2,4,5,6\};\\
      \left(1-\frac 1 p\right)^2\left(1+\frac 2 p-\frac 1{p^2}\right), &i=3.
    \end{cases}    
  \end{equation*}
\end{lemma}
\begin{proof}
  In the first case, by~\eqref{eq:torsor}, the last height condition is
  \begin{equation}\label{eq:last_height_1}
    |(\e_2\e_8^2+\e_4\e_5^3\e_6^2\e_8)/\e_1| \le B.
  \end{equation}
  Hence, by~\cite[Lemma~5.1(4)]{MR2520770},
  \begin{equation*}
    V_{1,1}(\e_1, \dots, \e_6;B) \ll \frac{B^{1/2}}{|\e_1\e_2|^{1/2}} =
    \frac{B}{|\e_1\e_2\e_3\e_4\e_5\e_6|} \left(\frac{B}{|\e_1\e_2\e_3^2\e_4^2\e_5^2\e_6^2|}\right)^{-1/2}.
  \end{equation*}
  In the second case, we use
  \begin{equation*}
    V_{2,1}(\e_1, \e_2, \e_5, \e_7; B) \ll \frac{B}{|\e_1\e_2\e_5\e_7|}.
  \end{equation*}
  In the third case, we use
  \begin{equation*}
    V_{3,1}(\e_1, \e_2, \e_5; B) \ll \frac{B}{|\e_1\e_2\e_5|}.
  \end{equation*}
  Therefore,~\cite[Proposition~4.3, Corollary~7.10]{MR2520770} gives
  the result in the first three cases. The final cases are
  similar to the second and third case.
\end{proof}

\section{Volume asymptotics}\label{sec:volumes}

We must show that the real integrals $V_{i,0}(B)$ in
Lemma~\ref{lem:surface-result-of-count} grow of order
$B(\log B)^{b_i-1}$. In the first and third case, this is more subtle than for
rational points.

\begin{lemma}\label{lem:adjust_V1}
  We have $|V_{1,0}'(B)-V_{1,0}(B)|\ll B (\log B)^4$, where
  \begin{equation*}
    V_{1,0}'(B) = \int_{\substack{|\e_2|,\dots,|\e_6|\ge 1\\
        \Hs_1'(\e_1,\dots,\e_6,\e_8) \le B}}
    \frac{\diff \e_1 \dots \diff \e_6 \diff \e_8}{|\e_1|}
  \end{equation*}
  with
  \begin{equation*}
    \Hs_1'(\e_1,\dots,\e_6,\e_8) = \max\left\{
      \begin{aligned}
        &|\e_2\e_3\e_4\e_5\e_6\e_8|,|\e_1\e_2\e_3^2\e_4^2\e_5^2\e_6^2|,\\
        &|\e_3\e_4^2\e_5^4\e_6^3|,|\e_2\e_8^2/\e_1|,|\e_2\e_3^2\e_4^2\e_5^2\e_6^2|
      \end{aligned}
    \right\}.
  \end{equation*}
\end{lemma}
\begin{proof}
  We must show that adding the condition
  $|\e_2\e_3^2\e_4^2\e_5^2\e_6^2| \le B$, removing the condition
  $|\e_1| \ge 1$, and replacing~\eqref{eq:last_height_1}
  by $|\e_2\e_8^2/\e_1| \le B$ in
  the integration domain changes the integral by $\ll B(\log B)^4$.

  Adding the condition $|\e_2\e_3^2\e_4^2\e_5^2\e_6^2| \le B$ does not change
  $V_{1,0}(B)$, since this inequality follows from $|\e_1| \ge B$ and the
  second height condition.
  Afterwards, we can remove the condition $|\e_1| \ge 1$ from $V_{1,0}(B)$ since
  this changes the integral by
  \begin{equation*}
    \int_{\substack{|\e_1|\le 1,\ |\e_2|,\dots,|\e_6|\ge 1\\
    \Hs_1(\e_1,\dots,\e_6,\e_8)\le B\\
    |\e_2\e_3^2\e_4^2\e_5^2\e_6^2| \le B}} \frac{\diff \e_1 \dots \diff \e_6
    \diff \e_8}{|\e_1|}
    \ll \int_{\substack{|\e_1| \le 1,\ |\e_2|,\dots,|\e_6|\ge 1\\
    |\e_2\e_3^2\e_4^2\e_5^2\e_6^2| \le B}}
    \frac{B^{1/2} \diff \e_1 \dots \diff \e_6}{|\e_1\e_2|^{1/2}},
  \end{equation*}
  where we estimate the integral over $\e_8$ as in the proof of
  Lemma~\ref{lem:surface-result-of-count}. Now we observe that the new
  condition $|\e_2\e_3^2\e_4^2\e_5^2\e_6^2| \le B$ together with
  $|\e_2|,\dots,|\e_6| \ge 1$ implies $|\e_2|,\dots,|\e_6|\le B$; all these
  conditions and $|\e_1|\le 1$ allow us to bound the error as required.

  Finally, we must replace~\eqref{eq:last_height_1} by $|\e_2\e_8^2/\e_1| \le B$. A comparison of $\Hs_1$ in
  Cox coordinates (Lemma~\ref{lem:heights}) with $H_1$ as in~\eqref{eq:height}
  motivates the transformation
  \begin{equation}\label{eq:change_of_coords}
    \e_8=\frac{B}{\e_2\e_3\e_4\e_5\e_6}x_0,\quad
    \e_1=\frac{B}{\e_2\e_3^2\e_4^2\e_5^2\e_6^2} x_2,
  \end{equation}
  which turns $\frac{\diff \e_1 \diff \e_8}{|\e_1|}$ into
  $\frac{B\diff x_0 \diff x_2}{|x_2\e_2\e_3\e_4\e_5\e_6|}$, and the
  transformation $\e_3 =\frac{B}{\e_4^2\e_5^4\e_6^3} x_3$, which turns
  $\frac{\diff \e_3}{|\e_3|}$ into $\frac{\diff x_3}{|x_3|}$. These
  transformations turn $\Hs_1(\e_1,\dots,\e_6,\e_8) \le B$ into
  \begin{equation*}
    |x_0|,|x_2|,|x_3|,|x_0(x_0+x_3)/x_2| \le 1.
  \end{equation*}
  Furthermore, they turn $|\e_3| \ge 1$ and
  $|\e_2\e_3^2\e_4^2\e_5^2\e_6^2| \le B$ (which imply
  $|\e_2\e_4^2\e_5^2\e_6^2| \le B$) into a condition $X_3 \le |x_3| \le X_3'$
  for certain $X_3$ and $X_3'$, whose values (depending on $\e_2,\e_4,\e_5,\e_5,B$)
  will not matter to us.
  Altogether, this shows that
  \begin{equation*}
    V_{1,0}(B) = \int_{\substack{|\e_2|, |\e_4|, |\e_5|, |\e_6| \ge 1\\|\e_2
        \e_4^2 \e_5^2 \e_6^2|\le B}} W(\e_2,\e_4,\e_5,\e_6,B)
    \frac{B\diff \e_2 \diff \e_4 \diff \e_5 \diff \e_6}{|\e_2\e_4\e_5\e_6|} +
    O(B(\log B)^4)
  \end{equation*}
  with
  \begin{equation*}
      W(\e_2,\e_4,\e_5,\e_6,B) = \int_{\substack{
          |x_0|,|x_2|,|x_3|,|x_0(x_0+x_3)/x_2| \le 1\\
      X_3 \le |x_3| \le X_3'
    }}
    \frac{\diff x_0 \diff x_2 \diff x_3}{|x_2 x_3|}\text{.}
  \end{equation*}
  Now the following Lemma~\ref{lem:treat_W} shows that we can replace
  $|x_0(x_0+x_3)/x_2|\le 1$ by $|x_0^2/x_2| \le 1$ with an error of $O(1)$. We
  plug this back into $V_{1,0}(B)$ and observe that the integral of
  $O(1)\cdot B/|\e_2\e_4\e_5\e_6|$ is $\ll B(\log B)^4$, while the inverse of
  our previous transformations turn the main term into
  $V_{1,0}'(B)$ since they turn $|x_0^2/x_2| \le 1$ into
  $|\e_2\e_8^2/\e_1| \le B$ and since we can remove the condition
  $|\e_2 \e_4^2 \e_5^2 \e_6^2|\le B$, which is implied by the others.
\end{proof}

To complete the proof of Lemma~\ref{lem:adjust_V1}, we show:

\begin{lemma}\label{lem:treat_W}
  We have $W(\e_2,\e_4,\e_5,\e_6,B) = W'(\e_2,\e_4,\e_5,\e_6,B)+O(1)$, where
  \begin{equation*}
    W'(\e_2,\e_4,\e_5,\e_6,B) = \int_{\substack{
          |x_0|,|x_2|,|x_3|,|x_0^2/x_2| \le 1\\
      X_3 \le |x_3| \le X_3'
    }}
    \frac{\diff x_0 \diff x_2 \diff x_3}{|x_2 x_3|}\text{.}
  \end{equation*}
\end{lemma}
\begin{proof}
  As a first step, we integrate over $x_2$ to get
    \begin{equation}\label{eq:W-step-1}
      W(\e_2,\e_4,\e_5,\e_6,B) = \int_{\substack{
        |x_0(x_0+x_3)| \le 1 \\
        |x_0|,|x_3| \le 1 \\
        X_3 \le |x_3| \le X_3'
      }} (-2\log |x_0| - 2 \log |x_0+x_3|) \frac{\diff x_0 \diff x_3}{|x_3|}
    \end{equation}
    and shall integrate the two terms individually.
    
    To determine the integral over the first one, we remove the condition
    $|x_0(x_0+x_3)| \le 1$, introducing an error of at most
    \[
    |R_1(\e_2,\e_4,\e_5,\e_6,B)| \le 4 \int_{\substack{
      x_0,|x_3| \le 1, x_0\ge 0\\
      |x_0(x_0+x_3)| \ge 1 \\
    }}
    -\log x_0 \frac{\diff x_0 \diff x_3}{|x_3|}
    \]
    by using the symmetry in the signs of $x_0$ and $x_3$. The last inequality implies that $x_3$ has a distance of at least $1/|x_0|$ (which is $\geq 1$) from $-x_0$. Since $x_0>0$ and $x_3>-1$, it cannot be smaller, and thus $-x_0+1/x_0 \le x_3 \le 1$ holds. We thus get
    \begin{align*}
      |R_1(\e_2,\e_4,\e_5,\e_6,B)|
      &\ll \int_{0\leq x_0 \leq 1} -\log x_0
      \left(\int_{-x_0+\frac{1}{x_0}\le x_3 \le 1}  \frac{\diff x_3}{|x_3|} \right) \diff x_0 \\
      & \ll \int_{0\leq x_0 \leq \frac{\sqrt{5}-1}{2}} \left| \log x_0 \log \left(-x_0 + \frac{1}{x_0}\right)\right| \diff x_0 \ll 1\text{.}
    \end{align*}
    We can now integrate the first term in~\eqref{eq:W-step-1} over $x_0$ and get
     \begin{equation}\label{eq:W-term-1}
      \int_{\substack{
        |x_0|,|x_3|, |x_0(x_0+x_3)| \le 1 \\
        X_3 \le |x_3| \le X_3'
      }}
      -2\log |x_0| \frac{\diff x_0 \diff x_3}{|x_3|}
      = \int_{\substack{
        |x_3| \le 1 \\
        X_3\le |x_3| \le X_3'
      }}
      4 \frac{\diff x_3}{|x_3|} + O(1)\text{.}
    \end{equation}

    To treat the second term, we begin with a change of variables $x_0'=x_0+x_3$ and add the condition $|x_0'|\le 1$, introducing an error of at most
    \begin{equation*}
    |R_2(\e_2,\e_4,\e_5,\e_6,B)|  \le \int_{\substack{
      |x_0'-x_3|,|x_3|, |x_0'(x_0'-x_3)| \le 1 \\
      x_0' > 1
    }}
     4 \log x_0' \frac{\diff x_0' \diff x_3}{|x_3|}\text{,}
    \end{equation*}
    again using the symmetry of the integral. The third condition implies $|x_3-x_0'| \leq 1/|x_0'| < 1$---that is, $x_0' - 1/x_0' < x_3$---and thus we get
    \begin{align*}
      |R_2(\e_2,\e_4,\e_5,\e_6,B)| & \ll
      \int_{x_0' > 1} \log x_0'
      \left(\int_{x_0'-\frac{1}{x_0'}}^{1}  \frac{\diff x_3}{|x_3|} \right) \diff x_0' \\
      & \ll
      \int_{1< x_0' \le 2}  \log x_0'
      \abs{\log\left(x_0'-\frac{1}{x_0'}\right)} \diff x_0' \ll 1\text{.}
    \end{align*}
    (For the second inequality, note that $x_0' - 1/x_0'\le1$ implies $x_0'\le 2$.) Thus, the second term of~\eqref{eq:W-step-1} is
    \[
    \int_{\substack{
      |x_0'(x_0'-x_3)| ,
      |x_0'|,|x_0'-x_3|,|x_3| \le 1 \\
      X_3 \le |x_3| \le X_3'
    }} - 2 \log |x_0'| \frac{\diff x_0' \diff x_3}{|x_3|} +O(1)\text{.}
    \]
    The condition $|x_0'(x_0'-x_3)|\le 1$ is implied by the second and third condition, so we can remove it. Removing $|x_0'-x_3|\le 1$ introduces an error of at most
    \[
      |R_3(\e_2,\e_4,\e_5,\e_6,B)| \le
       \int_{\substack{
        x_0',|x_3| \le 1\\
        |x_0'-x_3| > 1\\
        x_0'\ge 0
        }} - 2 \log x_0' \frac{\diff x_0' \diff x_3}{|x_3|}
    \]
    by the symmetry of the integral.
    The conditions imply $-1 \le x_3 \le x_0'-1$ and thus
    \begin{align*}
      |R_3(\e_2,\e_4,\e_5,\e_6,B)|
      &\ll
      \int_{0 \le x_0' \le 1} - \log x_0'
      \left(\int_{-1}^{x_0'-1}  \frac{\diff x_3}{|x_3|} \right) \diff x_0'\\
      & \ll \int_{0 \le x_0' \le 1} \log x_0'
      \log|x_0'-1|\diff x_0' \ll 1\text{.}
    \end{align*}
    Thus, the integral of the second summand of~\eqref{eq:W-step-1} is
    \begin{equation}\label{eq:W-term-2}
    \int_{\substack{
      |x_0'|,|x_3| \le 1 \\
      X_3 \le |x_3| \le X_3'
    }} - 2 \log |x_0'| \frac{\diff x_0' \diff x_3}{|x_3|} +O(1) =
    \int_{\substack{
      |x_3| \le 1 \\
      X_3 \le |x_3| \le X_3'}} 4 \frac{\diff x_3}{|x_3|} +O(1)\text{.}
\end{equation}
Since
\begin{equation}\label{eq:int_x0x2}
  \int_{|x_0|,|x_2|,|x_0^2/x_2| \le 1} \frac{\diff x_0 \diff x_2}{|x_2|} = 8,
\end{equation}
adding~\eqref{eq:W-term-1} and~\eqref{eq:W-term-2} yields the desired result.
\end{proof}

\begin{lemma}\label{lem:adjust_V3}
  We have $|V_{3,0}'(B)-V_{3,0}(B)|\ll B (\log B)^2$, where
  \begin{equation*}
    V_{3,0}'(B)= \int_{\substack{|\e_1|,|\e_2|,|\e_5|\ge 1,\\\Hs_3'(\e_1,\e_2,\e_5,\e_8) \le B}}
    \frac{\diff \e_1 \diff \e_2 \diff \e_5 \diff \e_8}{|\e_1|}
  \end{equation*}
  with
  \begin{equation*}
    \Hs_3'(\e_1,\e_2,\e_5,\e_8)=\max\{|\e_2\e_5\e_8|,|\e_1\e_2\e_5^2|,|\e_5^4|,\abs{\e_1^2\e_2^2}\}.
  \end{equation*}
\end{lemma}
\begin{proof}
  The difference that we must estimate is the integral over
  \[
    \abs{\e_2\e_5\e_8}, \abs{\e_1\e_2\e_5^2}, \abs{\e_5^4},
    \abs{(\e_2\e_8^2+\e_5^3\e_8)/\e_1} \le B \le \abs{\e_1^2\e_2^2}.
  \]
  Using the condition $\abs{(\e_2\e_8^2+\e_5^3\e_8)/\e_1} \le B$
  in~\cite[Lemma~5.1(4)]{MR2520770}, we have
  \begin{equation*}
    |V_{3,0}'(B)-V_{3,0}(B)| \ll \int_{\substack{|\e_1|,|\e_2|,|\e_5|\ge
        1\\|\e_1\e_2\e_5^2| \le B}} \frac{B^{1/2} \diff\e_1 \diff\e_2 \diff\e_5}{|\e_1\e_2|^{1/2}}.
  \end{equation*}
  The remaining conditions imply $|\e_1|,|\e_2|\le B$. Now the result follows
  by integrating first over $|\e_5| \le (B/|\e_1\e_2|)^{1/2}$ and then over $1
  \le |\e_1|,|\e_2| \le B$.
\end{proof}

\begin{lemma}\label{lem:V0}
  We have
  \begin{align*}
    V_{1,0}'(B) &= 2^5 C_1 B (\log B)^5,\quad
    V_{2,0}(B) = 2^3 C_2 B (\log B)^4,\quad
    V_{3,0}'(B) = 2^2 C_3 B (\log B)^3,\\
    V_{4,0}(B) &= 2 C_4 B (\log B)^2,\quad \text{and}\ \quad
    V_{5,0}(B) = 2^2 C_5 B (\log B)^3
  \end{align*}
  with
  \begin{align*}
      C_1
      & = 8 \vol\bigwhere{
        (t_2,\dots, t_6) \in \RRnn^5
      }{
        &t_2+2t_3+2t_4+2t_5+2t_6 \le 1,\\
        &t_3+2t_4+4t_5+3t_6 \le 1
      } = \frac{13}{4320},
      \\
      C_2
      & = 4 \vol\where{
        (t_1,t_2,t_5,t_7) \in \RRnn^4
      }{
        2t_1+2t_2 \le 1,\
        4t_5+2t_7 \le 1
      } = \frac{1}{32},
      \\
      C_3
      & = 4 \vol\where{
        (t_1,t_2,t_5) \in \RRnn^3
      }{
        2t_1+2t_2 \le 1,\
        4t_5 \le 1
      } = \frac{1}{8},
      \\
      C_4
      & = 4 \vol\{(t_1,t_2) \in \RRnn^2 \mid t_1+t_2 \le 1\} = 2,\quad \text{and}\\
      C_5
      & = 4 \vol\where{
        (t_1,t_5,t_7) \in \RRnn^3
      }{
        2t_1 \le 1,\
        t_1+2t_5+t_7 \le 1
      } = \frac{7}{24}.
  \end{align*}
\end{lemma}
\begin{proof}
  Again we apply the coordinate change~\eqref{eq:change_of_coords}, which shows that
  \begin{equation*}
    V_{1,0}'(B)=  \int_{\substack{
      |\e_2|, |\e_3|, |\e_4|, |\e_5|, |\e_6| \ge 1\\
      |\e_3 \e_4^2 \e_5^4 \e_6^3 |, |\e_2 \e_3^2 \e_4^2 \e_5^2\e_6^2|\le B
    }}
    \frac{B\diff \e_2 \diff \e_3 \diff \e_4 \diff \e_5 \diff
      \e_6}{|\e_2\e_3\e_4\e_5\e_6|}\cdot
    \int_{|x_0|,|x_2|,|x_0^2/x_2| \le 1} \frac{\diff x_0 \diff x_2}{|x_2|}
    \text{.}
  \end{equation*}
  The integral over $x_0,x_2$ is $8$ by~\eqref{eq:int_x0x2}.
  Restricting to positive $\e_i$ introduces a factor of $2^5$. Substituting
  $\e_i = B^{t_i}$ turns $\diff\e_i/\e_i$ into $\log B \diff t_i$, and we thus
  arrive at
  \begin{equation*}
    V_{1,0}'(B)  =2^5 \int_{\substack{
        t_2,t_3,t_4,t_5,t_6 \ge 0\\
        t_3 + 2 t_4 + 4 t_5 + 3 t_6 \le 1\\
        t_2 + 2 t_3 + 2 t_4 + 2 t_5 + 2 t_6 \le 1
      }} 8 B(\log B)^5\diff t_2 \diff t_3 \diff t_4 \diff t_5 \diff t_6 .
  \end{equation*}
  This integral can be interpreted as the volume of a polytope, which we
  compute using Magma.

  For the second case, using the first height condition yields
  \begin{align*}
    V_{2,0}(B) &= 2\int_{|\e_1|,|\e_2|,|\e_5|,|\e_7| \ge 1,\
                 |\e_1^2\e_2^2|,|\e_5^4\e_7^2|\le B} \frac{B \diff\e_1
                 \diff\e_2 \diff\e_5 \diff\e_7}{|\e_1\e_2\e_5\e_7|}.
  \end{align*}
  We proceed as in the first case; here, we can compute the volume by
  hand. The final two cases are analogous.

  For the third case, 
  we observe that $|\e_1\e_2\e_5^2|$ can be ignored in the definition of $\Hs_3'$
  since it is the geometric average of $|\e_1^2\e_2^2|$ and $|\e_5^4|$. Now
  the computation is very similar to the second case.
\end{proof}

Plugging this into Lemma~\ref{lem:surface-result-of-count} (after applying
Lemma~\ref{lem:adjust_V1} and Lemma~\ref{lem:adjust_V3} in the first and third
case) completes the proof of Theorem~\ref{thm:main_concrete}.

\section{The leading constant}\label{sec:leading-constant}

We show that Theorem~\ref{thm:main_concrete} can be abstractly formulated as
Theorem~\ref{thm:main_abstract}.
Part of the leading constants~\eqref{eq:c_fin} are $p$-adic Tamagawa volumes
$\tau_{(\tS,D_i),p}(\tUs_i(\ZZ_p))$ as defined
in~\cite[2.1.10, 2.4.3]{MR2740045}.  These measures are similar to the usual
Tamagawa volumes studied in the context of rational points, except for factors
$\norm{1_{D_i}}_p$ that are constant and equal to $1$ on the set of $p$-adic integral
points at almost all places (in fact, at all finite places in our cases). Over
the reals, the analogous volumes, when evaluated on the full space of real points, would
be infinite. Instead, \emph{residue measures} $\tau_{i,D_A,\infty}$ supported
on minimal strata $D_A(\RR)$ of the boundary divisors appear in the leading
constant~\eqref{eq:c_infty}, cf.~\cite[2.1.12]{MR2740045}. These can be interpreted as a density
function for the set of integral points (100\% of which are in arbitrarily
small real-analytic neighborhoods of the boundary; hence, a density function
has to be supported on the boundary), cf.\ \cite[3.5.8]{MR2999313}, or the
leading constant of an asymptotic expansion of the volume of \emph{height
  balls} with respect to $\tau_{(\tS,D_i),\infty}$,
cf.~\cite[Theorem~4.7]{MR2740045}.

In addition, we have to compute factors $\alpha_{i,A}$ as in~\eqref{eq:alpha} (cf.~\cite{wilsch-toric}), similar to Peyre's in the case of rational
points~\cite{MR2019019}. Again, there is one of these factors associated with any
minimal stratum $A$ of the boundary.

\smallskip

In order to compute the Tamagawa volumes, we work with the chart
\begin{align*}
  f\colon V'=\tS \setminus \VV(\e_1\e_2\e_3\e_4\e_5\e_6) &\to \AAA^2_\QQ,\\
  (\e_1:\e_2:\e_3:\e_4:\e_5:\e_6:\e_7:\e_8:\e_9)
  &\mapsto
  \left(\e_7 \cdot \frac{\e_5^2\e_6}{\e_1\e_2\e_3},\e_8\cdot \frac{1}{\e_1\e_3\e_4\e_5\e_6}\right)
\end{align*}
and its inverse $g\colon \AAA^2_\QQ \to \tS$,
\[
  (x,y) \mapsto
  (1:1:1:1:1:1:x:y:-x-y)\text{.}
\]
Note that the two elements
\[
\e_7 \cdot \frac{\e_5^2\e_6}{\e_1\e_2\e_3}\quad \text{and}\quad \e_8\cdot \frac{1}{\e_1\e_3\e_4\e_5\e_6}
\]
have degree $0$ in the field of fractions of the Cox ring. The rational map they define is thus invariant under the torus action and descends to $\tS$.
\begin{lemma}\label{lem:finite-volumes-sets}
  The images of the sets of $p$-adic integral points are 
  \begingroup
  \allowdisplaybreaks{}
    \begin{align*}
      f(\tUs_1(\ZZ_p)\cap V'(\QQ_p))
      &= \{(x,y) \in \QQ_p^2 \mid |x| \ge 1 \text{ or } |xy^2|\ge 1 \}, \\
      f(\tUs_2(\ZZ_p)\cap V'(\QQ_p))
      &= \{(x,y) \in \QQ_p^2 \mid |y| \le 1\text{ or } |xy^2| \le 1 \text { or } |x+y|\le 1\} \\
      &=\{|y|\le 1\} \cup \{|y| > 1,\ |xy^2| \le 1\} \cup \{|y| > 1,\ |x+y|\le 1\},\\
      f(\tUs_3(\ZZ_p)\cap V'(\QQ_p))
      &= f(\tUs_1(\ZZ_p)\cap V'(\QQ_p)) \cap f(\tUs_2(\ZZ_p)\cap V'(\QQ_p))\\
      &=\{|y|\le 1,\ |x|\ge 1\} \cup \{|y| > 1,\ |xy^2|=1\} \cup \{|y| > 1,\ |x+y|\le 1\},\\
      f(\tUs_4(\ZZ_p)\cap V'(\QQ_p)) 
      &= \{\abs{x}\ge 1, \abs{y}\le 1\} \cup \{\abs{y}>1, \abs{x+y}\le 1\}, \qquad \text{and} \\
      f(\tUs_5(\ZZ_p)\cap V'(\QQ_p)) 
      &= \{\abs{x},\abs{y}\le 1\} \cup \{\abs{y}>1, \abs{xy^2}\le 1\} \cup \{\abs{y}>1, \abs{x+y}\le 1\}.  
    \end{align*}
  \endgroup
  Here, the unions are disjoint.
\end{lemma}
\begin{proof}
  Consider the image $(x,y)$ of an integral point
  $\pi(\e_1,\dots,\e_9) \in \tUs_1(\ZZ_p)$.
  Assume $|x|< 1$. Then $\e_5 \not\in\ZZ_p^\times$ or $\e_6 \not\in \ZZ_p^\times$ (since $\e_7\in\ZZ_p^\times$).
  In both cases, the coprimality conditions imply $\e_8\in\ZZ_p^\times$, and thus
  $|xy^2|=|\e_7\e_8^2/ \e_1^3 \e_2 \e_3^3 \e_4^2 \e_6 | \ge 1$.

  On the other hand, let us consider a point $(x,y)$ in the above set and
  construct an integral point $(\e_1,\dots,\e_9)$ on the torsor with $f(\pi(\e_1,\dots,\e_9))=(x,y)$.
  If $|x| < 1$, we distinguish two cases for $|y|$:
  \begin{enumerate}[label = (\roman*)]
    \item If $1/|x|^{1/2} \le |y| < 1/|x|$, let $\e_5=xy$, $\e_6=1/xy^2$, $\e_9=-1-x/y$, and the remaining coordinates be $1$. Then $\e_9 \in -1 + p\ZZ_p\subset \ZZ_p^\times$ since $|x/y|\le |x|^{1/2} < 1$, and thus the coprimality conditions are satisfied.
    \item If $1/|x|\le |y|$, let $\e_4=1/xy$, $\e_6=x$, $\e_9=-1-x/y$, and let all the other coordinates be $1$.
    Since $|x/y| \le |x|^2 < 1$, we again have $\e_9 \in -1+p\ZZ_p\subset \ZZ_p^\times$, and thus the coprimality conditions hold.
  \end{enumerate}
  If $|x| \ge 1$, we distinguish three cases for $|y|$.
  \begin{enumerate}[label = (\roman*)]
    \item If $|y| < 1$, let $\e_2=1/x$, $\e_8=y$, $\e_9=-1-y/x$, and the remaining coordinates be $1$. Then $\e_9\in-1+p\ZZ_p \subset \ZZ_p^\times$, since $|y/x| < 1$.
    \item If $1 \le |y| < |x|$, let $\e_3=1/y$, $\e_2=y/x$, $\e_9=-1-y/x$, and the remaining coordinates be $1$. Again, we have $|y/x| < 1$, so that $\e_9\in-1+p\ZZ_p\subset \ZZ_p^\times$.
    \item Finally, if $|x|\le |y|$, let $\e_3=1/x$, $\e_4=x/y$, $\e_1=-1-x/y$, and the remaining coordinates be $1$. If $|y| > |x|$, we have $\e_1\in-1+p\ZZ_p\subset \ZZ_p^\times$; if $|x|=|y|$, we have $\e_4\in\ZZ_p^\times$.
    In both cases, the coprimality conditions on the torsor are satisfied.
  \end{enumerate}

  \smallskip

  We now turn to $\tUs_2$. Let $(x,y)$ be in the image of the set of integral points. If $|y|>1$, we have either $|\e_5|<1$ or $|\e_1|<1$. In the first case, we get
  $|xy^2|=|\e_7\e_8^2/\e_1^3\e_2|=|\e_7|\le 1$ (since all other variables have to be units);
  for the second case, we note that
  \[
    x+y = \frac{\e_4\e_5^3\e_6^2\e_7 + \e_2\e_8}{\e_1\e_2\e_3\e_4\e_5\e_6}
    = - \frac{\e_1\e_9}{\e_1\e_2\e_3\e_4\e_6},
  \]
  and thus $|x+y|=|\e_9|\le 1$ (since all other variables have to be units).

  On the other hand, let $(x,y)$ be in the set on the right hand side
  in the statement of the lemma. We want to construct an integral
  point on the torsor lying above $(x,y)$.  If $|y|\le 1$ and
  $|x|\le1$, let $\e_8=y$, $\e_7=x$, $\e_9=-x-y$, and the remaining
  variables be $1$, which satisfies the coprimality-conditions.  If
  $|y|\le 1$ and $|x|>1$, let $\e_8=y$, $\e_2=1/x$, $\e_9=-1-y/x$, and
  the remaining variables be $1$. Then
  $\e_9\in-1-p\ZZ_p\subset\ZZ_p^\times$, so $(\e_1,\dots,\e_9)$ is
  integral.  Let now $|y|>1$. If $|xy^2|\le 1$, let $\e_5=1/y$,
  $\e_7=xy^2$, $\e_9=-1-xy$, and the remaining variables be $1$;
  again, $\e_9\in\ZZ_p^\times$.  Finally, if $|x+y|\le 1$, let
  $\e_1=1/x$, $\e_9=-x-y$, $\e_8=-\e_1\e_9-1$, and the remaining
  variables be $1$.  Then $\e_8\in\ZZ_p^\times$, so
  $(\e_1,\dots,\e_9)$ is integral, and, since
  $\e_8/\e_1=(-\e_1\e_9-1)/\e_1 = x+y-x=y$, it indeed lies above
  $(x,y)$.
  For the disjoint union description of $\tUs_2$, we just have to observe that $|y|>1$ and
  $|xy^2| \le 1$ implies $|x|=|y|^{-2}<1$, while $|y|>1$ and $|x+y|\le 1$
  implies $|x|=|y|>1$.

  The third set consists of points that are integral with respect to
  both $Q_1$ and $Q_2$. Therefore, we obtain it as the intersection of the
  previous two sets. For the description of $\tUs_3$ as a disjoint union, we start with
  the one of $\tUs_2$ and intersect each set with
  $\tUs_1$. Here, $|y| \le 1$ implies $|x|\ge 1$ since otherwise
  $|xy^2|<1$. Furthermore, $|y|>1$ and $|xy^2|\le 1$ implies $|x|<1$;
  hence, $|xy^2|\ge 1$ must hold. Finally, $|y|>1$ and $|x+y|\le 1$
  implies $|x|=|y|>1$.

  The final two cases are analogous.
\end{proof}

\begin{lemma}
  Let $v$ be a place of $\QQ$. For the measures $\tau_{(\tS,D_i),v}$ defined in~\cite[2.4.3]{MR2740045}, we have
  \begingroup
  \allowdisplaybreaks{}
  \begin{align*}
    \diff f_*\tau_{(\tS,D_1),v} &=
    \frac{1}{|x|\max\{|y|,1,|x|,|y(y+x)|\}} \diff x \diff y, \\
    \diff f_*\tau_{(\tS,D_2),v} &=
    \frac{1}{\max\{|xy|,1,|x^2|\}} \diff x \diff y,  \\
    \diff f_*\tau_{(\tS,D_3),v} &=
    \frac{1}{|x|\max\{|y|,1,|x|,M(x,y)\}} \diff x \diff y,\\
    \diff f_*\tau_{(\tS,D_4),v} &=
    \frac{1}{\abs{x}\max\{\abs{y}, 1, \abs{x}\}} \diff x \diff y, \qquad \text{and}\\
    \diff f_*\tau_{(\tS,D_5),v} &=
    \frac{1}{\max\{\abs{xy}, 1, \abs{x}\}} \diff x \diff y,
  \end{align*}
  \endgroup
  where
  \[
    M(x,y) = \min\left\{
      \frac{\abs{y(x+y)}}{\abs{x^3}},
      \frac{\abs{x+y}}{\abs{x}},
      \abs{y(x+y)},
      \abs{x^{-1}}\right\},
  \]
  and all absolute values are $\abs{\cdot}=\abs{\cdot}_v$.
\end{lemma}
\begin{proof}
  In the first case, we have
  \begin{equation}\label{eq:def-tau-S-D1}
    \diff f_*\tau_{(\tS,D_1),v} =
    \|(\diff x \wedge \diff y) \otimes 1_{E_7} \|_{\omega_{\tS}(D),v}^{-1} \diff x \diff y.
  \end{equation}
  To make sense of this, we need a metric on the log-canonical bundle,
  not just on a line bundle isomorphic to it. To this end, we consider the
  isomorphism between the canonical bundle $\omega_{\tS}$ and the line bundle
  whose meromorphic sections are elements of degree $\omega_{\tS}$ of
  the field of fractions of Cox ring that maps $\diff x \wedge \diff y$ to
  $1/\e_1^2\e_2^2\e_3^3\e_4^2\e_6$; in addition, we consider the isomorphisms
  between $\Os(E_i)$ and the line bundles whose sections are elements of the
  Cox ring mapping $1_{E_i}$ to $\e_i$. Together, these induce an isomorphism
  from each $\omega_{\tS}(D_1)$ to the line bundle whose sections are
  functions of the Cox ring of degree $\omega_{\tS}(D_1)$, and we can pull
  back the adelic metric we constructed along this isomorphism (and similarly
  for the log-canonical bundles in the remaining cases).  In Cox coordinates,
  the norm in~\eqref{eq:def-tau-S-D1} at a point $\ee$ is
  \begin{equation}\label{eq:norm-dx-dy-1E7}
    \frac{|\e_1^2\e_2^2\e_3^3\e_4^2\e_6|}{|\e_7| \max\{|\e_2\e_3\e_4\e_5\e_6\e_8|,|\e_1\e_2\e_3^2\e_4^2\e_5^2\e_6^2|,|\e_3\e_4^2\e_5^4\e_6^3\e_7|,|\e_8\e_9|\}}.
  \end{equation}

  In the second case, we can analogously determine the norm
  \[
  \|(\diff x \wedge \diff y) \otimes 1_{E_3} \otimes 1_{E_4} \otimes 1_{E_6}\|_{\omega_{\tS}(D_2),v}^{-1}
  \]
  in Cox coordinates:
  \begin{equation}\label{eq:norm-dx-dy-case-2}
    \frac{1}{|\e_3\e_4\e_6|}
    \frac{|\e_1^2\e_2^2\e_3^3\e_4^2\e_6|}{ \max\{|\e_2\e_5\e_7\e_8|,|\e_1^2\e_2^2\e_3^2\e_4|,|\e_4\e_5^4\e_6^2\e_7^2|\}}.
  \end{equation}

  In the third case, the norm
  \[
  \|(\diff x \wedge \diff y) \otimes 1_{E_3} \otimes 1_{E_4} \otimes 1_{E_6} \otimes 1_{E_7}\|_{\omega_{\tS}(D_3),v}^{-1}
  \]
  at a point $\ee$ in Cox coordinates is
  \begin{equation}\label{eq:norm-dx-dy-case-3}
    \frac{1}{|\e_3\e_4\e_6\e_7|}
    \frac{|\e_1^2\e_2^2\e_3^3\e_4^2\e_6|}{ \max\{|\e_2\e_5\e_8|,|\e_1\e_2\e_3\e_4\e_5^2\e_6|,|\e_4\e_5^4\e_6^2\e_7|, M_0(\ee)\}}
  \end{equation}
  with
  \begin{equation*}
    M_0(\ee) = \frac{\abs{\e_1^3 \e_2^3 \e_3^2 \e_4 \e_8 \e_9}}{\max_{s \in B} \{\abs{s(\ee)}\}}.
  \end{equation*}
  Then $M_0(g(x,y))=M(x,y)$ as above
  after removing terms that can never contribute to the minimum.
  
  In the remaining two cases, the norms of interest are
  \begin{align*}
    &\frac{ 1 }{ \abs{\e_3\e_4\e_5\e_6\e_7} }
    \frac{|\e_1^2\e_2^2\e_3^3\e_4^2\e_6|}{\max\{\e_2\e_8, \e_1\e_2\e_3\e_4\e_5\e_6, \e_4\e_5^3\e_6^2\e_7\}}
    \quad \text{and}\\
    & \frac{1}{\abs{\e_2\e_3\e_4\e_6}}
    \frac{|\e_1^2\e_2^2\e_3^3\e_4^2\e_6|}{\max\{\abs{\e_5\e_7\e_8},\abs{\e_1^2\e_2\e_3^2\e_4},\abs{\e_1\e_3\e_4\e_5^2\e_6\e_7}\}},
  \end{align*}
  respectively.
\end{proof}

\begin{lemma}\label{lem:surface-finite-volumes}
  Let $p$ be a finite prime. Then
  \begin{equation*}
    \tau_{(\tS,D_i),p}(\tUs_i(\ZZ_p))=
    \begin{cases}
      1+\frac{6-\#D_i}{p}, &i=1,2,4,5;\\
      1+\frac{2}{p}-\frac{1}{p^2}, &i=3.
    \end{cases}
  \end{equation*}
\end{lemma}
\begin{proof}
  We compute 
  \begin{equation}\label{eq:tau-p-abstract}
    \tau_{(\tS,D_i),p}(\tUs_i(\ZZ_p))=\int_{f(\tUs_i(\ZZ_p)\cap V'(\QQ_p))} \diff f_*\tau_{(\tS,D_i),p}
  \end{equation}
  for $i\in\{1,\dots,5\}$.
  For $i=1$, the previous two lemmas transform this into
  \[
    \int_{\substack{x,y\in\QQ_p\\|x|\ge 1 \text{ or }|xy^2| \ge 1}}
    \frac{1}{|x|\max\{|y|,1,|x|,|y(y+x)|\}} \diff x \diff y.
  \]
  Subdividing the domain of integration into the regions with $|x|>|y|$, $|x|=|y|$, and $|x|<|y|$ in order to simplify the denominator, we get
  \begin{equation}\label{eq:tau-p-divided}
    \begin{aligned}
      &\int_{\substack{|y|<|x|\\|x|\ge 1}} \frac{1}{|x|\max\{|x|,|xy|\}}\diff x \diff y
      + \int_{\substack{|y|=|x|\\|x|\ge 1}} \frac{1}{|x|\max\{|x|,|y(y+x)|\}} \diff x \diff y \\
      &\qquad +
      \int_{\substack{|x| < |y|\\ |xy^2| \ge 1}} \frac{1}{|xy^2|} \diff x \diff y
  \end{aligned}
  \end{equation}
  after simplifying the description of the domains ($|x|<1$ would imply $|xy^2| \le|x|^3<1$ in the first two cases;
  $|y|^2<1/|x|$ would imply $|y|^2 <1/|x| \le 1 \le |x|^2 < |y|^2 $ in the third case).

  The first of the integrals in~\eqref{eq:tau-p-divided} is
  \begin{align}
      &\int_{|x|\ge 1} \frac{1}{|x|^2}
      \int_{|y|<|x|} \frac{1}{\max\{1,|y|\}} \diff y \diff x
      = \int_{|x|\ge 1} \frac{1}{|x|^2}
      \left(\frac{1}{p} + \int_{1\le |y| < |x|} \frac{1}{|y|} \diff y\right)
      \diff x \nonumber\\
      & \quad = \int_{|x|\ge 1} \frac{1}{|x|^2}\left(\frac{1}{p} + \left(1-\frac{1}{p}\right)|v(x)| \right) \diff x
      = \frac{1}{p}
      + \sum_{\delta\ge 0} \left(1-\frac{1}{p}\right)^2\frac{\delta}{p^\delta} = \frac{2}{p} \text{,} \label{eq:tau-p-aux1}
  \end{align}
  while the second integral is
  \begin{equation}\label{eq:tau-p-aux2}
    \int_{\substack{|y+x|\le \frac{1}{p} \\ |x|\ge 1}} \frac{1}{|x|^2}
    + \int_{\substack{|y+x| \ge 1 \\ |x|\ge 1, |y| = |x|}} \frac{1}{|xy(x+y)|}.
  \end{equation}
  The first integral in~\eqref{eq:tau-p-aux2} is $\frac{1}{p}\int_{|x|\ge 1}\frac{1}{|x|^2}\diff x = \frac{1}{p}$. Turning to the second one, we note that $|x| = |y|$ is implied by the ultrametric triangle inequality if $|x+y|<|x|$.
  The set of $y\in\QQ_p$ with $|x+y|=|x|$ and $|y|=|x|$ has volume
  $|x|-2|x|/p$ since the two sets $\{y \mid |y-0|< |x|\}$ and $\{y \mid |y+x|<
  |x|\}$ have volume $|x|/p$ and are disjoint (because $|y|<|x|$ implies $|y+x|=|x|$). We thus get
  \begin{align*}
    &\int_{|x| \ge 1} \frac{1}{|x|^2}
    \left(\sum_{0\leq \delta< |v(x)|} \left(1-\frac{1}{p}\right)\frac{p^\delta}{p^\delta} + \left(1-\frac{2}{p}\right)\frac{|x|}{|x|}\right) \diff x \\
    & \quad = \int_{|x| \ge 1} \frac{1}{|x|^2}
    \left( \left(1-\frac{1}{p}\right)|v(x)| + \left(1-\frac{2}{p}\right) \right)\diff x = \frac{1}{p} + 1-\frac{2}{p} = 1 - \frac{1}{p}\text{,}
  \end{align*}
  computing the integral over $x$ similarly as in~\eqref{eq:tau-p-aux1}. The second integral in~\eqref{eq:tau-p-divided} thus evaluates to $1$.
  Finally, the third integral in~\eqref{eq:tau-p-divided} is
  \begin{align*}
    &\int \frac{1}{|y|^2} \int_{1/|y|^2 \le |x| < |y|} \frac{1}{|x|} \diff x \diff y
    = \int_{|y|\ge 1} \frac{1}{|y|^2} \sum_{-2|v(y)|\le \delta < |v(y)|} \left(1-\frac{1}{p}\right) \diff y \\
    & \quad = \int_{|y| \ge 1} \left(1-\frac{1}{p} \right) \frac{3|v(y)|}{|y|^2}
    = \frac{3}{p}\text{,}
  \end{align*}
  again computed analogously to the previous ones.
  Adding the three terms in $\eqref{eq:tau-p-divided}$, we arrive at
  our claim for $i=1$.

  \smallskip

  For $i=2$, we get
  \begin{align*}
    &\int_{|y|\le 1\text{, } |xy^2| \le 1 \text {, or } |x+y|\le 1}
    \frac{1}{|x|\max\{|y|,|x^{-1}|,|x|\}} \diff x \diff y
     \\
    & \quad = \int_{|y|\le 1} \frac{1}{\max\{1,|x^2|\}} \diff x \diff y
    + \int_{\substack{|y|> 1 \\ |x+y|\le 1}} \frac{1}{|y^2|} \diff x \diff y
    + \int_{\substack{|y|> 1 \\ |x|\le 1/|y|^2}} \diff x \diff y
  \end{align*}
  for the integral~\eqref{eq:tau-p-abstract} (since $|x|=|y|$ in the second case). The first integral is then
  \[
    1 + \int_{|x|>1}\frac{1}{|x^2|} \diff x = 1+\frac{1}{p},
  \]
  while the other two integrals are
  \[
    \int_{|y|>1}\frac{1}{|y^2|} \diff y = \frac{1}{p}.
  \]
  
  For $i=3$, we compute~\eqref{eq:tau-p-abstract} to be
  \begin{align*}
    &\int_{f(\tUs_3(\ZZ_p)\cap V'(\QQ_p))}
    \frac{1}{|x|\max\{|y|,1,|x|,M(g(x,y))\}} \diff x \diff y
     \\
    & \quad = \int_{|y|\le 1,\ |x| \ge 1} \frac{1}{|x^2|} \diff x \diff y
    + \int_{\substack{|y|> 1 \\ |x+y|\le 1}} \frac{1}{|y^2|} \diff x \diff y
    + \int_{\substack{|y|> 1 \\ |x|= 1/|y|^2}} \diff x \diff y
  \end{align*}
  (again using that $|x|=|y|$ in the second case). The first integral is then
  \[
    \int_{|x| \ge 1}\frac{1}{|x^2|} \diff x = 1,
  \]
  while the second one is
  \[
    \int_{|y|>1}\frac{1}{|y^2|} \diff y = \frac{1}{p},
  \]
  and the third one is 
  \begin{equation*}
    \left(1-\frac 1 p\right)\int_{|y|>1} \frac{1}{|y^2|} \diff y =
    \frac{1}{p} - \frac{1}{p^2}.
  \end{equation*}
  
  The final two cases are similar: For $i=4$, the integrals over the two disjoint sets in Lemma~\ref{lem:finite-volumes-sets} are $1$ and $p^{-1}$, respectively, while for $i=5$, the integrals over the three disjoint sets are $1$, $p^{-1}$, and $p^{-1}$, respectively.
\end{proof}

The remaining parts of the constant are associated with maximal faces of the
Clemens complex.  Recall from Section~\ref{sec:expected} and
Figure~\ref{fig:clemens_complex} that the Clemens complex of the geometrically
irreducible divisor $D_1$ consists of just one vertex
$E_7$.
For $D_2$, we have three vertices corresponding to its components, and two
$1$-simplices $A_1=\{E_3,E_4\}$ and $A_2=\{E_4,E_6\}$ between the intersecting
exceptional curves (Figures~\ref{fig:clemens_complex} and~\ref{fig:dynkin-diagram}). The Clemens complex for
$D_3=D_1+D_2$ is the disjoint union of the previous two cases; its maximal
dimensional faces are again $A_1$ and $A_2$. For $D_4$, they are
$A_1,\dots,A_4$, and for $D_5$, they are $A_1,A_2,A_5$ (Figure~\ref{fig:dynkin-diagram}).

For a face $A$ of the Clemens
complexes associated with $D_i$, we set $D_A=\bigcap_{E\in A} E$, and
$\Delta_{i,A}= D_i - \sum_{E \in A} E$.
For a maximal-dimensional face $A$ of a Clemens complex, the adjunction
isomorphism and a metric on the log-canonical bundle $\omega_{\tS}(D_i)$
induce a metric on the bundle
$\omega_{D_A}\otimes \Os_{\tS}(\Delta_{i,A})|_{D_A}$ on $D_A$.  Since $A$ is
maximal, the canonical section $1_{\Delta_{i,A}}$ does not have a pole on
$D_A$, so since $D_A(\RR)$ is compact, the norm
$\norm{1_{\Delta_{i,A}}|_{D_A}}_{\Os_{\tS}(\Delta_{i,A})|_{D_A},\infty}$ is
bounded on $D_A(\RR)$ for any metric. Hence,
\[
  \norm{\omega \otimes 1_{\Delta_{i,A}}|_{\Os_{\tS}}}_{\omega_{D_A} \otimes \Os_{\tS}(\Delta_{i,A})|_{D_A}, \infty}^{-1} \abs{\omega}
  = \norm{1_{\Delta_{i,A}}|_{D_A}}_{\Os_{\tS}(\Delta_{i,A})|_{D_A},\infty} \tau_{D_A, \infty}
\]
(where the equality holds for any choice of metrics on $\omega_{D_A}$ and
$\Os_{\tS}(\Delta_{i,A})|_{D_A}$ compatible with the one on their tensor
product) defines a finite measure on $D_A(\RR)$, independent of the choice of
a form $\omega\in\omega_{D_A}$. We further normalize this measure with a
factor $c_\RR^{\#A}=2^{\#A}$, call it \emph{residue measure}, and denote it by
$\tau_{i,D_A,\infty}$. See~\cite[2.1.12, 4.1]{MR2740045} for details on this
construction.

\begin{lemma}\label{lem:arch-volumes}
  We have
  \[
    \tau_{1,E_7,\infty}(E_7(\RR)) = 8
    \quad \text{and} \quad
    \tau_{i,D_A,\infty}(D_A(\RR)) = 4
  \]
  for $i\in \{2,\dots,5\}$ and every maximal-dimensional face $A$ of the Clemens complex for $D_i$. 
\end{lemma}
\begin{proof}
  Following~\cite[2.1.12]{MR2740045}, we can compute the unnormalized Tamagawa volume of $E_7$ by integrating
  \[
    \|\diff y \|_{\omega_{E_7},\infty}^{-1}=
    \lim_{x\to 0} \left(|x|  \| (\diff x \wedge \diff y) \otimes 1_{E_7}\|_{\omega_{\tS}(E_7), \infty}^{-1} \right)\text{.}
  \]
  Again evaluating~\eqref{eq:norm-dx-dy-1E7} in the image of $(x,y)$, we get the volume
  \begin{align*}
    \tau^\prime_{1,E_7,\infty}(E_7(\RR))&=\int_\RR \lim_{x\to 0}\frac{|x|}{|x|\max\{
    1,|x|,|y|,|y(y+x)|\}}\diff y
    \\
    &= \int_\RR \frac{1}{\max\{1,|y^2|\}} \diff y =4,
  \end{align*}
  which we normalize by multiplying with $c_\RR=2$.

  For the second and third case, we work in neighborhoods of the two
  intersection points $D_{A_1}=E_3\cap E_4$ and
  $D_{A_2}=E_4 \cap E_6$. The Tamagawa measures on these points are
  simply real numbers.  In order to compute them, we consider the
  charts
  \begin{align*}
    g'  &\colon \AAA^2_\QQ \to \tS,\ (a,b) \mapsto (1:1:a:b:1:1:1:1:-1-b)\ \quad \text{and}\\
    g'' &\colon \AAA^2_\QQ \to \tS,\ (c,d) \mapsto (1:1:1:c:1:d:1:1:-1-cd).
  \end{align*}
  We have $x=1/a=d$, $y=1/ab=1/cd$ for these charts. Since $\|\diff x \wedge \diff y \|= |\det(J_{f\circ g'})|\| \diff a\wedge \diff b\|$, we can use~\eqref{eq:norm-dx-dy-case-2} to compute the norms
  \begin{align*}
    \|(\diff a \wedge \diff b) \otimes 1_{E_3} \otimes 1_{E_4} \otimes 1_{E_6}\|_{\omega_{\tS}(D_2), \infty}
    &= \max\{|a^3b^2|,|ab|,|ab^2|\} \quad \text{and}\\
    \|(\diff c \wedge \diff d) \otimes 1_{E_3} \otimes 1_{E_4} \otimes 1_{E_6}\|_{\omega_{\tS}(D_2), \infty}
    & = \max\{|c^2d|,|cd|,|c^2d^3|\}\text{.}
  \end{align*}
  Analogously to the first case, we now arrive at
  \[
    \tau^\prime_{2,D_{A_1},\infty} = \lim_{(a,b)\to (0,0)} \frac{|ab|}{\max\{|a^3b^2|,|ab|,|ab^2|\}} = 1,
  \]
  and, similarly, $\tau^\prime_{2,D_{A_2},\infty}=1$
  for the unnormalized measures on the points $D_{A_i}(\RR)$, which we multiply with $c_\RR^2=4$.

  In the third case, using the same change of variables and~\eqref{eq:norm-dx-dy-case-3}, we get
  \[
    \norm{(\diff a \wedge \diff b) \otimes 1_{E_3} \otimes 1_{E_4} \otimes 1_{E_6} \otimes 1_{E_7}}_{\omega_{\tS}(D_3), \infty}
    = \abs{ab}\max\{1,\abs{ab}, \abs{b}, M_0(g'(a,b))\}
  \]
  with $M_0(g'(ab))\to 0$, as $(a,b)\to (0,0)$,
  whence $\tau^\prime_{3,D_{A_1},\infty}=1$ for the unnormalized measure.
  Finally,
  \[
    \norm{(\diff c \wedge \diff d) \otimes 1_{E_3} \otimes 1_{E_4} \otimes 1_{E_6} \otimes 1_{E_7}}_{\omega_{\tS}(D_3), \infty}
    = \abs{cd}\max\{1,\abs{cd}, \abs{cd^2}, M_0(g''(c,d))\},
  \]
  where again $M_0(g''(c,d))\to 0$, and we get
  $\tau^\prime_{3,D_{A_2},\infty}=1$ for the unnormalized
  measure. Again, we multiply both measures with $c_\RR^2=4$.
  
  The computations in the cases $i=4,5$ are analogous.
\end{proof}

These Tamagawa numbers are multiplied with rational numbers $\alpha_{i,A}$,
where $A$ is a maximal-dimensional face of the Clemens complex for $D_i$,
depending on the geometry of certain effective cones, as
in~\eqref{eq:alpha}. Following~\cite[\S~2.2]{wilsch-toric},
\begin{equation}\label{eq:tU_iA}
  \tU_{i,A} = X \setminus \bigcup_{\substack{E_j\subset D_i,\\E_j \not\in
      A}}E_j
\end{equation}
is the complement of all boundary components not belonging to $A$, and
$\Eff \tU_{i,A} \subset (\Pic\tU_{i,A})_\RR$ is its
effective cone; all volume functions are normalized as in~\cite[Remark~2.2.9~(iv)]{wilsch-toric}.

\begin{lemma}\label{lem:alpha}
  We have
  \begin{align*}
    \alpha_{1,E_7}
      & =\vol\bigwhere{(t_2,\dots, t_6) \in \RRnn^5}{&t_2+2t_3+2t_4+2t_5+2t_6 \le 1\\
      &t_3+2t_4+4t_5+3t_6 \le 1}=\frac{13}{34560},\\
    \alpha_{2,A_1} &= \vol\where{(t_1,t_2,t_5,t_7) \in \RRnn^4}{t_1+t_2 \le
                     2t_5+t_7,\ 4t_5+2t_7 \le 1}=1/256,\\
    \alpha_{2,A_2}&= \vol\where{(t_1,t_2,t_5,t_7) \in \RRnn^4}{t_1+t_2 \ge 2t_5+t_7,\ 2t_1+2t_2 \le 1 }=1/256,\\
    \alpha_{3,A_1} &= \vol\where{(t_1,t_2,t_5) \in \RRnn^3}{t_1+t_2 \le 2t_5,\ 4t_5 \le 1}=1/96,\\
    \alpha_{3,A_2}&= \vol\where{(t_1,t_2,t_5) \in \RRnn^3}{t_1+t_2 \ge 2t_5,\
                    2t_1+2t_2 \le 1}=1/48,\\
    \alpha_{4,A_1}&= 0,\\
    \alpha_{4,A_2}&= \vol\where{(t_1,t_2) \in \RRnn^2}{t_1+t_2 \le 1/2}=1/8,\\
    \alpha_{4,A_3}&= \vol\where{(t_1,t_2) \in \RRnn^2}{1/2 \le t_1+t_2 \le 2/3}=7/72,\\
    \alpha_{4,A_4}&= \vol\where{(t_1,t_2) \in \RRnn^2}{2/3 \le t_1+t_2 \le 1}=5/18,\\
    \alpha_{5,A_1}&= \vol\where{(t_1,t_5,t_7) \in \RRnn^3}{t_1 \le 2t_5+t_7,\
                    4t_5+2t_7 \le 1}=1/48,\\
    \alpha_{5,A_2}&= \vol\where{(t_1,t_5,t_7) \in \RRnn^3}{t_1 \ge 2t_5+t_7,\
                    2t_1 \le 1}=1/96,\quad \text{and}\\
    \alpha_{5,A_5}&= \vol\where{(t_1,t_5,t_7) \in \RRnn^3}{t_1+2t_5+t_7\le 1,\
                    4t_5+2t_7 \ge 1}=1/24.
  \end{align*}
\end{lemma}
\begin{proof}
  To compute $\alpha_{i,A}$, we choose $j_0 \in \{1,\dots,7\}$ such that
  $E_{j_0} \in A$ and such that the classes of $E_j$ for
  $j \in \{1,\dots,7\} \setminus \{j_0\}$ form a basis of $\Pic \tS$. The
  latter holds for $j_0 \in \{1,2,3,6,7\}$ since the data
  in~\cite{MR3180592} shows that $\Pic \tS$ has rank $6$ and is generated by
  the classes of the negative curves $E_1, \dots, E_7$, where
  \begin{equation}\label{eq:pic_relation}
    E_1+E_2+E_3-2E_5-E_6-E_7
  \end{equation}
  is a principal divisor. An inspection of Figure~\ref{fig:dynkin-diagram}
  shows that $E_{j_0} \in A$ for some $j_0 \in \{3,6,7\}$.

  Hence, there are unique linear combinations $\sum_{j \ne j_0} a_j E_j$ of
  class $\omega_{\tS}(D_i)^\vee$ and $\sum_{j \ne j_0} b_j E_j$ of the same
  class as $E_{j_0}$; the coefficients $a_j,b_j \in \ZZ$ can be computed
  using~\eqref{eq:pic_relation} and the fact that $2E_1+2E_2+3E_3+2E_4+E_6$ has
  anticanonical class by~\eqref{eq:antican_sections}. For the following
  computations, it is useful to know that
  \begin{equation}\label{eq:log_anticanonical_classes}
    2E_4+4E_5+3E_6+2E_7,\quad 2E_1+2E_2+3E_3+2E_4,\quad E_1+E_2+2E_3+2E_4+2E_5+2E_6
  \end{equation}
  have class $\omega_{\tS}(E_{j_0})^\vee$ for $j_0=3,6,7$, respectively
  (expressed without using $E_{j_0}$).
  
  Let $J = \{j \in \{1,\dots,7\} \mid E_j \subset D_i,\ E_j \notin A\}$ and
  $J' = \{1,\dots,7\} \setminus (J \cup \{j_0\})$. By definition~\eqref{eq:tU_iA},
  \begin{equation*}
    \Pic \tU_{i,A} = (\Pic \tS)/\langle E_j \mid j \in J\rangle;
  \end{equation*}
  hence, a basis is given by the classes of $E_j$ for $j \in J'$ modulo the
  classes of $E_j$ for $j \in J$, and its effective cone is generated by the
  classes of $E_j$ for $j \in J' \cup \{j_0\}$ modulo the classes of $E_j$ for
  $j \in J$. Working with the dual basis, we obtain
  \begin{equation*}
    \alpha_{i,A} = \vol\left\{(t_j) \in \RRnn^{J'} \mid \sum_{j \in J'} a_j t_j =1,\
    \sum_{j \in J'} b_j t_j\ge 0\right\}.
  \end{equation*}
  If $A=\{E_{j_0},E_{j_1}\}$ is a $1$-simplex, then $j_1 \in J$, and the next
  step is to eliminate the variable $t_{j_1}$ using using the equation, which
  gives a description of $\alpha_{i,A}$ as the volume of a polytope in
  $\RRnn^{J_i}$ with $J_i$ as in~\eqref{eq:def_Ji} defined by two inequalities.

  In the first case, we have $\tU_{1,E_7}=\tS$ and corresponding effective cone $\Eff\tS$,
  whose dual is the nef cone of $\tS$.
  Working with the dual basis of the classes of $E_1,\dots,E_6$ and using~\eqref{eq:pic_relation}
  and~\eqref{eq:log_anticanonical_classes} for $E_7$
  and $\omega_{\tS}(D_1)^\vee$, we obtain
  \begin{equation*}
    \alpha_1=\vol\bigwhere{(t_1,\dots, t_6) \in \RRnn^6}{&t_1+t_2+t_3-2t_5-t_6 \ge 0\\
     &t_1+t_2+2t_3+2t_4+2t_5+2t_6 = 1}
  \end{equation*}
  and eliminate $t_1$.

  In the second case, there are two constants $\alpha_{2,A_i}$ associated with
  the maximal faces $A_1=\{E_3,E_4\}$ and $A_2=\{E_4,E_6\}$ of the Clemens
  complex.  The subvarieties used in their definition are
  $\tU_{2,A_1}=\tS \setminus E_6$ and $\tU_{2,A_2}=\tS \setminus E_3$. In the first case, we have
  $J=\{6\}$, choose $j_0=3$, and obtain $J'=\{1,2,4,5,7\}$.
  Therefore, the Picard group of $\tU_{A_1}$ is
  $(\Pic\tS)/\langle E_6 \rangle$ with a basis is given by the classes of
  $E_1,E_2,E_4,E_5,E_7$ modulo $E_6$, and its effective cone is generated by
  the classes of $E_1,\dots,E_5,E_7$ modulo $E_6$.
  Since $E_3$ has the same class as $-E_1-E_2+2E_5+E_6+E_7$ in $\Pic \tS$
  by~\eqref{eq:pic_relation}, while $E_4+4E_5+2E_6+2E_7$ has class
  $\omega_{\tS}(D_2)^\vee$ by~\eqref{eq:log_anticanonical_classes}, we obtain
  (working modulo $E_6$)
  \begin{equation*}
    \alpha_{2,A_1} = \vol\bigwhere{(t_1,t_2,t_4,t_5,t_7) \in \RRnn^5}{&-t_1-t_2+2t_5+t_7 \ge 0\\
     &t_4+4t_5+2t_7 = 1}
  \end{equation*}
  and eliminate $t_4$.

  The computation of $\alpha_{2,A_2}$ is similar. Here, we choose $j_0=6$, and
  our basis is given by the classes of $E_1,E_2,E_4,E_5,E_7$ modulo $E_3$.
  The divisor $E_6$ has the same class as $E_1+E_2+E_3-2E_5-E_7$, while
  $2E_1+2E_2+2E_3+E_4$ has class $\omega_{\tS}(D_2)^\vee$. Therefore,
  \begin{equation*}
    \alpha_{2,A_2} = \vol\bigwhere{(t_1,t_2,t_4,t_5,t_7) \in \RRnn^5}{&t_1+t_2-2t_5-t_7 \ge 0\\
     &2t_1+2t_2+t_4 = 1};
  \end{equation*}
  again, we eliminate $t_4$.

  The further cases are analogous. The only exceptional case is the
  computation of $\alpha_{4,A_1}$. Working with $J=\{5,6,7\}$, $j_0=3$ and
  $J'=\{1,2,4\}$, a similar computation as for $\alpha_{2,A_1}$ shows 
  \begin{equation*}
   \alpha_{4,A_1} = \vol\bigwhere{(t_1,t_2,t_4) \in \RRnn^5}{&-t_1-t_2 \ge 0,\\
     &t_4 = 1},
  \end{equation*}
  which clearly has volume $0$ in the hyperplane $t_4=1$.
\end{proof}

\begin{remark}\label{rmk:obstruction}
This last phenomenon $\alpha_{4,A_1}=0$ is an instance of the obstruction described in~\cite[Theorem~2.4.1~(i)]{wilsch-toric}: The regular function
\[
  s = \frac{\e_1 \e_2 \e_3}{\e_5 \e_6 \e_7}  
\]
on $\tUs_4$ is also regular on $\tU_{4,A_1}$. On the one hand, this regular function induces the relation $[E_1]+[E_2] = 0$ in $\Pic(\tU_{i,A_4})$, while both classes on the left are nonzero; this makes the pseudo-effective cone fail to be strictly convex, and the resulting polytope has volume $0$. On the other hand, $s$ vanishes on $D_{4,A_1} = \{\e_3=\e_4=0\}$, so that $\abs{s}<1$ on a sufficiently small real analytic neighborhood $W$ of $D_{4,A_1}$; but $W$ is integral on $\tUs_4(\ZZ)$ and nonzero on $V=\tS \setminus (E_1+\cdots +E_7)$, so $\abs{s}\ge 1$ on the set $V(\QQ) \cap \tUs_4(\ZZ)$ counted by $N_4$. It follows that any sufficiently small analytic neighborhood of $D_{4,A_1}(\RR)$ cannot contribute to $N_4$, which is reflected by the vanishing of the corresponding part of the expected leading constant. 
\end{remark}

\begin{lemma}\label{lem:arch-constants}
  For $i \in \{1,\dots,5\}$, the archimedean contributions to the expected
  constants are
  \begin{equation*}
    c_{i,\infty} = \sum_A \alpha_{i,A} \tau_{i,D_A,\infty}(D_A(\RR)) = C_i,
  \end{equation*}
  where the sum runs through the maximal faces $A$ of the Clemens complex,
  with $C_i$ as in Lemma~\ref{lem:V0}.
\end{lemma}
\begin{proof}
  This follows from Lemma~\ref{lem:arch-volumes} and
  Lemma~\ref{lem:alpha}. For $i=2,\dots,6$, we observe that the polytopes of
  volumes $\alpha_{i,A}$ in Lemma~\ref{lem:alpha} fit together to the one
  appearing in the description of $C_i$ in
  Lemma~\ref{lem:V0}.
\end{proof}

We conclude by noting that the classes of $E_3,E_4,E_6,E_7$ in $\Pic\tS$ are
linearly independent; hence, $\rk \Pic\tU_i=\rk\Pic\tS - \#D_i$ (with $\#D_i$
as in Lemma~\ref{lem:param}).
This observation, Lemma~\ref{lem:surface-finite-volumes}, and Lemma~\ref{lem:arch-constants}
allow us to reformulate Theorem~\ref{thm:main_concrete} as
Theorem~\ref{thm:main_abstract} for $i \in \{1, \dots, 5\}$. For the final case, we equip the log-anticanonical bundle $\omega_{\tS}(D_6)^\vee$ with the metric pulled back from $\omega_{\tS}(D_5)^\vee$ along the isomorphism~\ref{eq:symmetry}; since all constructions in this section
are invariant under metric preserving isomorphisms, the theorem follows for $i=6$.

\end{document}